\documentclass[10pt,reqno]{amsart}
\usepackage[scale=0.75, centering, headheight=14pt]{geometry}
\usepackage{fancyhdr}
\usepackage{amssymb,amsmath,enumerate}
\usepackage{graphics}
\usepackage{graphicx}
\usepackage{caption}
\usepackage{float}
\usepackage{subcaption}
\usepackage[linesnumbered,boxruled,commentsnumbered]{algorithm2e}
\usepackage{comment,color}

\usepackage{amsthm}

\usepackage{mathtools}

\usepackage{cleveref}

\theoremstyle{plain}
\newtheorem{theorem}{Theorem}

\theoremstyle{remark}
\newtheorem{remark}{Remark}

\DeclareMathOperator*{\argmin}{arg\,min}

\numberwithin{equation}{section}

\graphicspath{{./figures/}{.}}

\begin{document}

\title[Error indicators and adaptivity for the Monge-Amp\`ere equation]{Error indicators and adaptivity for a least-squares method applied to the Monge-Amp\`ere equation}

\author[A. Caboussat]{Alexandre Caboussat}
\address{
Geneva School of Business Administration (HEG-GEN\`EVE), 
University of Applied Sciences and Arts Western Switzerland (HES-SO), 
1227 Carouge, Switzerland, 
Email : \texttt{alexandre.caboussat@hesge.ch} 
}

\author[A. Peruso]{Anna Peruso}

\address{
Institute of Mathematics, Ecole Polytechnique F\'ed\'erale de Lausanne, 
1015 Lausanne, Switzerland, 
Email : \texttt{anna.peruso@epfl.ch}
and
Geneva School of Business Administration (HEG-GEN\`EVE), 
University of Applied Sciences and Arts Western Switzerland (HES-SO), 
1227 Carouge, Switzerland, 
Email : \texttt{anna.peruso@hesge.ch} 
}

\author[M. Picasso]{Marco Picasso}
\address{
Institute of Mathematics, Ecole Polytechnique F\'ed\'erale de Lausanne, 
1015 Lausanne, Switzerland, 
Email : \texttt{marco.picasso@epfl.ch}
}


\date{\today}

\begin{abstract}
We introduce novel a posteriori error indicators for a nonlinear least-squares solver for smooth solutions of the Monge-Ampère equation on convex polygonal domains in $\mathbb{R}^2$. The theoretical analysis is carried out for smooth solutions. At each iteration, our iterative scheme decouples the problem into (i) a pointwise nonlinear minimization problem and (ii) a linear biharmonic variational problem. For the latter, we derive an equivalence to a biharmonic problem with Navier boundary conditions and solve it via mixed piecewise-linear finite elements. Reformulating this as a second‐order system, we derive a priori and a posteriori $\mathbb{P}_1$ finite element error estimators and we design a robust adaptive mesh refinement strategy. Numerical tests, including both smooth and nonsmooth examples, confirm that errors in different norms scale appropriately. Finally, we demonstrate the effectiveness of our a posteriori indicators in guiding mesh refinement.
\end{abstract}

\date{\today}

\maketitle


\section{Introduction}\label{sec:intro}
In its classical formulation, the elliptic Monge-Amp\`ere equation reads
\begin{equation*}
 \text{\rm det}\:  D^2 u(x) = f(x, u, \nabla u) \quad x\in \Omega,   
 \end{equation*}
where $\Omega \subset \mathbb{R}^2$ denotes an open set, $u: \Omega \rightarrow \mathbb{R}$ is a convex function and $D^2 u$ its Hessian matrix, and $f: \Omega \times \mathbb{R} \times \mathbb{R}^2 \rightarrow \mathbb{R}^+$ is a given positive function \cite{dephilippis}. This fully nonlinear partial differential equation (PDE) governs the product of the eigenvalues of the Hessian matrix of $u$, unlike the elliptic Poisson equation $-\Delta u = f$, which governs the sum of the eigenvalues.  Assuming $f\geq 0$, the convexity of the solution $u$ is crucial, since it makes the equation (degenerate) elliptic, i.e. monotone with respect to the Hessian in the positive-semidefinite order, which is a necessary hypothesis for regularity results. Smoothness of $\Omega$ and $f$ are necessary to ensure existence of solutions in $C^2(\Bar{\Omega})$ \cite{dephilippis}. The Monge-Amp\`ere equation appears in various contexts, such as the prescribed Gaussian curvature equation (also known as the Minkowski problem). It also finds applications in fields like meteorology (modeling air and water flows in the troposphere) and fluid mechanics (determining wind velocity fields given a pressure field) \cite{feng}. Moreover, Monge-Amp\`ere type equations play a pivotal role in the theory of regularity and singularity of optimal transport maps \cite{dephilippis,villani}. 

Due to the growing importance of this equation as a fundamental example of fully nonlinear PDEs with a wide range of applications, many numerical techniques have been developed in recent decades to approximate its solutions. Although one might naturally attempt to apply discretization methods that work well for linear and quasi-linear PDEs, such approaches are generally unsuitable for fully nonlinear second-order PDEs and integration by parts cannot be used to transfer \textit{hard-to-control} derivatives onto the test function to form a variational formulation in a weaker Sobolev space. Nevertheless, several Galerkin-based methods have been proposed. For instance, the $L^2$ projection method \cite{bohmer,brenner,brenner2}, the vanishing moment method \cite{neilan}, the nonvariational finite element method \cite{lakkis} and the augmented Lagrangian approach \cite{lagrangian} have all been successfully applied. In this work, we analyse the nonlinear least-squares method proposed in \cite{caboussat} and further developed in \cite{prins,yadav,dimitrios}. The method has been proposed to approximate solutions in $H^2(\Omega)$ to second order fully nonlinear PDEs and it is based on a least-squares formulation of the PDE and a decoupling of the nonlinearity and of the differential operator. This decoupling leads to a system where the nonlinear component is solved pointwise and the fourth-order linear PDE is addressed separately, with the overall solution iteratively obtained by alternating between these two subproblems until convergence is reached.

Previous approaches \cite{caboussat,dimitrios} have solved the linear fourth-order subproblem using a conjugate gradient algorithm in Sobolev spaces combined with a mixed $\mathbb{P}_1$ finite element approximation, the additional inner iterations required at each outer step made this approach computationally expensive. In this work, we propose a direct finite element solver for a fourth-order subproblem such that, at each iteration, the main computational effort reduces to solving two Poisson problems; this eliminates the conjugate gradient step used in \cite{caboussat,dimitrios}, makes the method simpler than, for instance, approaches based on Newton linearization, and significantly reduces the overall computational cost. To improve the approximation of the Hessian at each iteration, we employ a recovery technique based on a post-processed gradient, following the approach in \cite{hessianrecovery}. We also establish stability and error estimates for the local nonlinear problem and both \textit{a priori} and \textit{a posteriori} error estimates for the $\mathbb{P}_1$ finite element approximation of the fourth-order problem and the recovered Hessian on two-dimensional convex polygonal domains. Numerical experiments confirm that the same order of convergence extends to the full solution. For smooth test cases, we observe an $H^2$ convergence $\mathcal{O}(h)$, improving upon the results reported in previous studies \cite{caboussat}. For nonsmooth problems, our method yields consistent convergence results in the $L^2$ norm. Finally, we incorporate residual-based \textit{a posteriori} estimators to drive an adaptive mesh refinement strategy. The error indicator used proves to be efficient, and the resulting mesh refinement, by optimizing node placement, produces numerical approximations with significantly reduced errors. The strategy remains effective even for nonsmooth problems, demonstrating the robustness of the method.

This article is organized as follows. In \Cref{sec:ls}, we describe the splitting algorithm for the least-squares formulation of the Monge-Ampère equation \cite{caboussat}. \Cref{sec:biharmonic,sec:hessian} present two main contributions: a direct approximation of the fourth-order subproblem and a Hessian recovery strategy with \textit{a priori} and \textit{a posteriori} error estimates, while \Cref{ssec:nonlinear} addresses the stability of the nonlinear subproblem. In \Cref{sec:errorind}, we show how to combine these estimates to derive error indicators for the Monge-Ampère equation. Finally, \Cref{sec:num} validates the theoretical results through numerical experiments, including adaptive mesh refinement tests.

\section{Least-squares formulation and splitting algorithm for Monge-Amp\`ere equation}\label{sec:ls}
Let $\Omega\subset\mathbb{R}^2$ be a bounded, convex domain and let $\partial\Omega$ denote its boundary. Assume that $f\in L^1(\Omega)$ is positive and that $g\in H^{3/2}(\partial \Omega)$. The elliptic Dirichlet Monge-Amp\`ere problem is given by  
\begin{equation}\label{eq:prob}
\begin{cases}
\text{\rm det}\:D^2u = f \quad &\text{in } \Omega, \\
u = g \quad &\text{on } \partial\Omega,\\
\end{cases}
\end{equation}
where the unknown function $u$ is convex and $D^2u$ denotes its Hessian, i.e. $[D^2u]_{ij} =\frac{\partial^2 u}{\partial x_i \partial x_j}$. Among the various methods available for solving \eqref{eq:prob} in $H^2(\Omega)$, we advocate a nonlinear least-squares formulation that relies on the introduction of an additional auxiliary variable \cite{caboussat}. In order to do so, let us define $\mathbf{P} = D^2u$, with $\mathbf{P}\in L^2(\Omega,\mathbb{R}^{2\times 2})$, and rewrite \eqref{eq:prob} as 
\begin{equation}\label{eq:prob1}
\begin{cases}
\text{\rm det}\:\mathbf{P} = f \quad &\text{in } \Omega, \\
\mathbf{P} = D^2u\quad &\text{in } \Omega, \\
u = g \quad &\text{on } \partial\Omega.\\
\end{cases}
\end{equation}
Given that we look for the convex solution to \eqref{eq:prob}, we impose the additional constraint that $\mathbf{P}$ must be symmetric positive definite (henceforth, spd). If there exists a solution $u$ to \eqref{eq:prob} in $H^2(\Omega)$, then $(u,\mathbf{P})=(u,D^2u)$ is a solution to the reformulated problem \eqref{eq:prob1}. Moreover, $(u,\mathbf{P})$ is the minimizer of the following problem:
\begin{equation}
\label{eq:leastsq}
(u, \mathbf{P}) = \argmin_{\substack{v\in H^2(\Omega)\cap H^1_g(\Omega) \\ \mathbf{Q}\in  L^2(\Omega; \mathbb{R}^{2\times 2})}}\left\{ J(v, \mathbf{Q}),\quad\text{s.t. } \text{\rm det}\:\mathbf{Q} = f,\,\mathbf{Q}\text{ spd}\right\},
\end{equation} 
where the functional $J(\cdot,\cdot)$ is defined by
\begin{equation*}
J(v, \mathbf{Q}):=\dfrac{1}{2}\int_\Omega |D^2v - \mathbf{Q}|^2,
\end{equation*}
and $|\cdot|$ denotes the Frobenius norm and $H^1_g(\Omega):=\{v\in H^1(\Omega):\,v|_{\partial\Omega}=g\}$. Here, $J(v,\mathbf{Q})$ measures the $L^2$ distance between the Hessian of $v$ and the auxiliary variable $\mathbf{Q}$, while the nonlinearity is accounted for through the constraint $\text{\rm det}\:\mathbf{Q} = f$. If $u\in H^2(\Omega)$ is a solution to the reformulated problem \eqref{eq:prob1}, then $J(u,D^2u) = 0$ and $(u, D^2u)$ is a minimizer of the functional \eqref{eq:leastsq}. This approach, which reformulates a fully nonlinear PDE as a nonlinear least-squares problem, can also be applied to other first and second-order fully nonlinear PDEs, including the Jacobian equation, the Eikonal equation, and Pucci’s equation \cite{prins,yadav,lsjac,lsorthogonal,dimitrios_adaptive,pucci}, where the main modification lies in the definition of the nonlinear constraint.
\begin{remark}
Notice that if there exists a unique convex solution $u\in H^2(\Omega)$ to \eqref{eq:prob}, then the minimizer of \eqref{eq:leastsq} must also be unique. Otherwise, if there were another minimizer $(u_1,\mathbf{P}_1)$ with $J(u_1,\mathbf{P}_1) = 0$, then $(u_1,\mathbf{P}_1)$ would also solve \eqref{eq:prob1}, contradicting the assumed uniqueness of the solution to \eqref{eq:prob} and \eqref{eq:prob1}. This level of regularity is a standard assumption in the numerical analysis of fully nonlinear PDEs, e.g. \cite{brenner,brenner2,neilan,lakkis}. Conversely, if no solution $u\in H^2(\Omega)$ to \eqref{eq:prob} exists, existence and uniqueness of a minimizer for \eqref{eq:leastsq} remains an open question.
\end{remark}
In order to approximate the solution to \eqref{eq:leastsq}, we advocate for a splitting algorithm \cite{caboussat} that iteratively decomposes the minimization problem \eqref{eq:leastsq} into two subproblems. Specifically, given an initial function \( u^0 \in H^2(\Omega)\), for \( n \geq 0 \), we seek \( \mathbf{P}^n \) and \( u^{n+1} \) such that:
\begin{subequations}\label{eq:splitting}
\begin{align}
\label{eq:firstmin}\mathbf{P}^n =& \argmin_{\mathbf{Q}\in L^2(\Omega; \mathbb{R}^{2\times 2})}\left\{J(u^n,\mathbf{Q}),\quad\text{s.t. } \text{\rm det}\:\mathbf{Q} = f,\,\mathbf{Q}\text{ spd}\right\},\\
\label{eq:biharmonic} u^{n+1} = &\argmin_{v\in H^2(\Omega)\cap H^1_g(\Omega)}J(v,\mathbf{P}^n).
\end{align}
\end{subequations}
In this formulation, the nonlinearity of the constraint is isolated in the first subproblem \eqref{eq:firstmin}, while the second subproblem \eqref{eq:biharmonic} deals with the variational character of the problem. The first subproblem can be solved pointwise using a Lagrange multiplier argument, as detailed in \Cref{ssec:nonlinear}. Meanwhile, the second subproblem corresponds to a fourth-order differential problem; its numerical approximation is detailed in \Cref{sec:biharmonic}. Although a rigorous convergence proof for the sequence $(u^n,\mathbf{P}^n)$ converging to $(u,\mathbf{P})$ is not available yet, numerical results show that, with proper initialization, the iterative algorithm converges \cite{caboussat,dimitrios,dimitrios_adaptive}. 
\begin{remark}\label{rem:distance}
By definition of \eqref{eq:firstmin} and \eqref{eq:biharmonic}, we obtain: 
\begin{equation*}
0\leq J(u^{n+1} , \mathbf{P}^{n+1})\leq J(u^{n+1},\mathbf{P}^n)\leq J(u^n ,\mathbf{P}^n)\leq \dots \leq J(u^0 , \mathbf{P}^0),\quad \forall n\geq 0,
\end{equation*}
and thus $J(u^n,\mathbf{P}^n)$ converges when $n\to \infty$.
\end{remark}

\subsection{Initialization of the splitting algorithm}\label{ssec:initialization}
For the initialization of the algorithm, we assume that the eigenvalues of $D^2u^0$, denoted by $\lambda_1$ and $\lambda_2$, are close ($\lambda_1 \approx \lambda_2$) \cite{caboussat}. In that case, 
$$(\Delta u^0)^2=(\lambda_1 + \lambda_2)^2 \approx 4(\lambda_1 \lambda_2)=4f.$$
Then, in order to initialize $u^0$ we solve the following Poisson problem:
\begin{equation}\label{eq:initialization}
\begin{cases}
\Delta u^0 = 2\sqrt{f}\quad &\text{in }\Omega,\\
u^0 = g\quad &\text{on }\partial\Omega.
\end{cases}
\end{equation}
This initialization is commonly used in the literature, not only for nonlinear least-squares methods \cite{lakkis}.

\section{Approximation of the fourth-order problem \eqref{eq:biharmonic}}\label{sec:biharmonic}
The second subproblem in the splitting algorithm \eqref{eq:splitting} is a fourth-order biharmonic type variational problem, equivalent to:
$$u^{n+1} = \argmin_{v\in H^2(\Omega)\cap H^1_g(\Omega)}\int_\Omega\left\{ \dfrac{1}{2}|D^2v|^2 -\mathbf{P}^n:D^2v\right\}.$$
This formulation seeks a function $u$ whose Hessian matrix is the closest, in the $L^2$ sense, to a given symmetric tensor field $\mathbf{P}^n$. The associated Euler-Lagrange equation reads: 
\begin{equation}\label{eq:secondminn}
\text{Find }u^{n+1}\in H^2(\Omega)\cap H^1_g(\Omega):\,\int_\Omega D^2 u^{n+1} : D^2 v = \int_\Omega \mathbf{P}^n:D^2 v,
\end{equation}
for any $v\in H^2(\Omega)\cap H^1_0(\Omega)$. Since the bilinear form $a(u,v) = \int_\Omega D^2u : D^2v$ defines an inner product on $H^2(\Omega)\cap H^1_0(\Omega)$, problem \eqref{eq:secondminn} is well posed. In previous works \cite{caboussat,dimitrios}, problem \eqref{eq:secondminn} was approximated using a conjugate gradient algorithm in Hilbert spaces, based on the inner product \( \langle u, v \rangle_{H^2(\Omega)\cap H^1_0(\Omega)} = \int_{\Omega} \Delta u \, \Delta v \). However, this approach introduces an additional layer of iteration to an already computationally intensive algorithm. Specifically, each iteration requires solving two Poisson problems, and numerical experiments reported in \cite{caboussat} indicate that approximately 10 iterations are needed to achieve a tolerance of \(10^{-5}\), with the number of iterations increasing as the mesh is refined. 

In this work, we propose solving \eqref{eq:secondminn} using a direct finite element solver. This eliminates the need for an inner iterative loop. Besides improving numerical accuracy, this strategy also reduces computational costs by approximately an order of magnitude.
To approximate \eqref{eq:secondminn} using $\mathbb{P}_1$ mixed finite elements (as detailed in \Cref{ssec:fe}), we aim to reformulate the problem in terms of a system of two second-order equations. Let $\nu$ and $\tau$ denote the unit normal and tangent vectors to the boundary $\partial \Omega$. Assuming sufficient regularity of $u$ and $v$, integration by parts twice gives:
\begin{equation*}
\begin{split}
    \int_\Omega (D^2 u^{n+1} - \mathbf{P}^n) : D^2 v = &\int_{\partial \Omega} (D^2u^{n+1}-  \mathbf{P}^{n} ): (\nu \otimes \nu)\frac{\partial v}{\partial \nu} \\
    &+ \int_\Omega (\Delta^2 u^{n+1} - \text{\rm div}(\text{\rm div}(\mathbf{P}^n))) v,\\
\end{split}
\end{equation*}
for any $v\in H^2(\Omega)\cap H^1_0(\Omega)$, where $(\text{\rm div}(\mathbf{P}))_i := \sum_j\frac{\partial}{\partial x_j}\mathbf{P}_{ij}$. From now on, let assume that $\Omega$ is a convex polygon. Then, using the identity
\[
\Delta u^{n+1} = D^2 u^{n+1} : (\nu \otimes \nu) + D^2 u^{n+1} : (\tau \otimes \tau) \quad \text{on }\partial\Omega ,
\]
and relating the tangential part to the boundary data $u^{n+1} = g$ via
\begin{equation}\label{eq:bc}
\frac{d^2 g}{ds^2} = D^2 u^{n+1} : (\tau \otimes \tau) \quad \text{on }\partial\Omega,
\end{equation}
where $s$ is the arc-length parameter along $\partial \Omega$, we obtain the strong formulation of \eqref{eq:secondminn}:
\begin{equation}\label{eq:bilapla}
\begin{cases}
\Delta^2 u^{n+1} = \text{\rm div}(\text{\rm div}(\mathbf{P}^n))\quad &\text{in }\Omega,\\
\Delta u^{n+1}=\phi^n\quad &\text{on }\partial\Omega,\\
u^{n+1}=g\quad &\text{on }\partial\Omega,\\
\end{cases}
\end{equation}
where $\phi^n :=  \mathbf{P}^{n} : (\nu \otimes \nu) + \frac{d^2 g}{ds^2}$. By defining $\omega^{n+1} := -\Delta u^{n+1}$, we can reformulate \eqref{eq:bilapla} as two Poisson problems to be solved sequentially, see also \cite{sweers_corner,invalidity}. In particular, we first solve for $\omega^{n+1}$, independently of $u^{n+1}$, and subsequently recover $u^{n+1}$. Their weak formulation is as follows: find $(\omega^{n+1}, u^{n+1}) \in H^1_{-\phi^n}(\Omega) \times H^1_g(\Omega)$ such that 
\begin{equation}\label{eq:splitbihdiv}
\begin{cases}
\displaystyle\int_\Omega \nabla \omega^{n+1} \cdot \nabla \psi = -\int_\Omega \text{\rm div}(\mathbf{P}^n) \cdot \nabla \psi,  &\forall \psi \in H^1_0(\Omega),\\[2mm]
\displaystyle\int_\Omega \nabla u^{n+1} \cdot \nabla v = \int_\Omega \omega^{n+1} v,  &\forall v \in H^1_0(\Omega).
\end{cases}
\end{equation}
The solution to \eqref{eq:splitbihdiv} is equivalent to the solution to \eqref{eq:secondminn} if the domain is convex, as detailed in \cite{sweers_corner,invalidity}, where they also show that this property does not hold for domains with reentrant corners. The computational domains considered in this work are convex as it is a necessary condition to have smooth solutions to the Monge-Amp\`ere equation.

\subsection{$\mathbb{P}_1$ FE approximation of $(\omega^{n+1},u^{n+1})$}\label{ssec:fe}
In this section we introduce the finite element discretization of \eqref{eq:secondminn}. The key point is that \eqref{eq:secondminn} reduces to two successive Dirichlet Poisson problems thanks to the boundary conditions in \eqref{eq:bilapla} and the convexity of the domain. Therefore the same P1 finite element space can be employed for both Poisson subproblems, which make the approach computationally attractive by using fast existing solvers. By contrast, the classical clamped biharmonic problem, whose natural energy space is $H^2_0(\Omega)$, is more constrained at the discrete level and involves different boundary conditions. In primal formulations, the lowest-order discretizations already involve piecewise quadratics, such as the Morley element \cite{Morley1968}; alternatively, one may use $C^0$ interior penalty methods \cite{brennerbih}. Mixed formulations avoid the construction of globally $C^1$ spaces, but they still lead to coupled biharmonic systems rather than two independent Dirichlet Poisson problems \cite{ciarletraviart,monk,scholz}. As an alternative to this method proposed here, the method in \cite{guo2018} employs $C^0$ linear finite elements coupled with a Hessian recovery technique to solve problem \eqref{eq:secondminn}. It achieves similar convergence orders for the biharmonic problem.

For any $h>0$, let \(\mathcal T_h\) be a conforming, triangulation of \(\overline\Omega\) into triangles \(K\) of diameter \(h_K\leq h\) and assume that the mesh $\mathcal{T}_h$ is regular \cite{ciarlet}, i.e. there exists $\vartheta>0$ such that for any $K\in\mathcal{T}_h$
$$\frac{h_K}{\rho _K}\leq \vartheta,$$
where $\rho_K$ is the diameter of the largest ball inscribed in $K$. Define  
\[
V_h(\Omega) := \{ v_h \in C^0(\Bar{\Omega}) : v_h |_{K} \in \mathbb{P}_1, \, \forall K \in \mathcal{T}_h \} \subset H^1(\Omega),
\]  
and 
$$ V_{h,\alpha}(\Omega):=\{v\in V_h(\Omega): v|_{\partial\Omega}=\alpha_h\},$$
where $\alpha_h$ is an approximations of $\alpha$, e.g. the Lagrange interpolant if $\alpha\in H^{1/2}(\Omega)$. Now let $\mathbf{P}^n_h, g_h, \phi_h^n\in V_h(\Omega)$ be some approximations of $\mathbf{P}^n, g, \phi^n$, respectively, defined on the mesh $\mathcal{T}_h$. Details are given in \Cref{ssec:nonlinear}. We then seek \( (\omega_h^{n+1}, u_h^{n+1}) \in V_{h,-\phi^n}(\Omega)\times V_{h,g}(\Omega) \) such that  
\begin{equation}\label{eq:fe}
\begin{cases}
\displaystyle \int_\Omega \nabla \omega_h^{n+1} \cdot \nabla \psi_h = -\int_\Omega \text{\rm div}(\mathbf{P}^n_h) \cdot \nabla \psi_h, \quad &\forall \psi_h \in V_{h,0}(\Omega),\\[2mm]
\displaystyle \int_\Omega \nabla u_h^{n+1} \cdot \nabla v_h = \int_\Omega \omega_h^{n+1} v_h,  &\forall v_h \in V_{h,0}(\Omega) .
\end{cases}
\end{equation}
We define the discretization errors as
$\epsilon_h^{n+1} := \omega^{n+1} - \omega_h^{n+1}$ and $e_h^{n+1} := u^{n+1} - u_h^{n+1}$. Then, we derive \textit{a priori} estimates and residual-based \textit{a posteriori} bounds in in \Cref{lemma:apriori} and \Cref{lemma:aposteriori}, respectively.
\begin{theorem}\label{lemma:apriori}
Let $\Omega\subset\mathbb{R}^2$ be a convex polygon. Let $\mathbf{P}^n\in H^2(\Omega,\mathbb{R}^{2\times 2})$, and $g,\frac{d^2g}{ds^2} \in H^{3/2}(\partial\Omega)$. Then, the following estimates hold: 
\begin{subequations}
\begin{align}
\label{eq:apriorinablaepsilon}
\| \nabla \epsilon_h^{n+1}\|_{L^2(\Omega)}\lesssim& h\big(\|\mathbf{P}^n\|_{H^2(\Omega)}+\|\phi^n\|_{H^{\frac{3}{2}}(\partial\Omega)}\big) \\ \notag
& +\|\mathbf{P}^n-\mathbf{P}^n_h\|_{H^1(\Omega)} +\|\phi^n-\phi^n_h\|_{H^{1/2}(\partial\Omega)},\\
\label{eq:aprioriepsilon}
\|\epsilon_h^{n+1}\|_{L^2(\Omega)}\lesssim&h\| \nabla \epsilon_h^{n+1}\|_{L^2(\Omega)} \\ \notag
&+ \|\mathbf{P}^n-\mathbf{P}^n_h\|_{L^2(\Omega)} +\|\phi^n-\phi^n_h\|_{H^{-1/2}(\partial\Omega)},\\
\label{eq:apriorinablae}
\| \nabla e_h^{n+1}\|_{L^2(\Omega)}\lesssim& h\|\omega^{n+1}\|_{L^2(\Omega)}+h\|g\|_{H^{\frac{3}{2}}(\partial\Omega)}\\ \notag
& + \|\epsilon^{n+1}_h\|_{H^{-1}(\Omega)} +\|g-g_h\|_{H^{1/2}(\partial\Omega)} ,\\
\label{eq:apriorie}
\|e_h^{n+1}\|_{L^2(\Omega)}\lesssim& h\| \nabla e_h^{n+1}\|_{L^2(\Omega)} \\ \notag
&+ \|\epsilon^{n+1}_h\|_{H^{-1}(\Omega)} +\|g-g_h\|_{H^{-1/2}(\partial\Omega)} .
\end{align}
\end{subequations}
\end{theorem}

\begin{proof}
These estimates follow from standard regularity results for the Poisson equation and interpolation estimates \cite{brenner_fem}, applied successively to \(\omega^{n+1}\) and then to \(u^{n+1}\).
\end{proof}
Except for the terms measuring data‐mismatch ($\mathbf{P}^n-\mathbf{P}^n_h$, $\phi^n-\phi^n_h$ and $g-g_h$), these estimates imply that if $\mathbf{P}^n\in H^2(\Omega)$, then 
\[
\|\nabla\epsilon_h^{n+1}\|_{L^2(\Omega)}=\mathcal{O}(h),
\quad
\|\epsilon_h^{n+1}\|_{L^2(\Omega)}=\mathcal{O}(h^2).
\]
Similarly, for $u_h$, from \eqref{eq:apriorinablae}-\eqref{eq:apriorie} we obtain
\[
\|\nabla e_h^{n+1}\|_{L^2(\Omega)} = \mathcal{O}(h),\quad \|e_h^{n+1}\|_{L^2(\Omega)}=\mathcal{O}(h^2),
\]
given that $\|\epsilon^{n+1}_h\|_{H^{-1}(\Omega)}\leq \|\epsilon^{n+1}_h\|_{L^{2}(\Omega)}$. 
\begin{theorem}\label{lemma:aposteriori}
Let $\Omega\subset\mathbb{R}^2$ be a convex polygon. Let $\mathbf{P}^n\in H^2(\Omega,\mathbb{R}^{2\times 2})$, and $g,\frac{d^2g}{ds^2} \in H^{3/2}(\partial\Omega)$. Then the following estimates hold: 
\begin{subequations}
\begin{align}
\label{eq:h1erromega}
\| \nabla \epsilon_h^{n+1}\|_{L^2(\Omega)}\lesssim &\big( \sum_{K\in\mathcal{T}_h} \eta_K^2\big)^{1/2}\\ \notag
&+\|\mathbf{P}^n-\mathbf{P}^n_h\|_{H^1(\Omega)}+\|\phi^n - \phi_h^n\|_{H^{1/2}(\partial\Omega)},\\
\label{eq:l2erromega}\|\epsilon_h^{n+1}\|_{L^2(\Omega)}\lesssim& \big(\sum_{K\in\mathcal{T}_h} h_K^2\eta_K^2\big)^{1/2} \\ \notag
&+ \|\mathbf{P}^n-\mathbf{P}^n_h\|_{L^2(\Omega)}+\|\phi^n-\phi^n_h\|_{H^{-1/2}(\partial \Omega)},\\
\label{eq:h1erru}
\| \nabla e_h^{n+1}\|_{L^2(\Omega)}\lesssim &\big(\sum_{K\in\mathcal{T}_h} \hat{\eta}_K^2\big)^{1/2} \\ \notag
&+ \|\epsilon_h^{n+1}\|_{H^{-1}(\Omega)} +\|g - g_h\|_{H^{1/2}(\partial\Omega)},\\
\label{eq:l2erru}\|e_h^{n+1}\|_{L^2(\Omega)}\lesssim &\big(\sum_{K\in\mathcal{T}_h} h_K^2\hat{\eta}_K^2\big)^{1/2} \\ \notag
&+ \|\epsilon_h^{n+1}\|_{H^{-1}(\Omega)}+ \|g - g_h\|_{H^{-1/2}(\partial\Omega)},
\end{align}
\end{subequations}
with
\begin{align}\label{eq:etaK}
\eta_K = &h_K\|\text{\rm div}(\text{\rm div}(\mathbf{P}^n_h)+\nabla  \omega_h^{n+1})\|_{L^2(K)} \\ \notag
&+ h^{\frac{1}{2}}_K\|[(\text{\rm div}(\mathbf{P}^n_h)+\nabla  \omega_h^{n+1})\cdot n_K]\|_{L^2(\partial K)},
\end{align}
and
\begin{equation}\label{eq:etahatK}
\hat{\eta}_K = h_K\|\omega_h^{n+1}+\Delta u_h^{n+1}\|_{L^2(K)} + h^{\frac{1}{2}}_K\|[\nabla  u_h^{n+1}\cdot n_K]\|_{L^2(\partial K)}.
\end{equation}
\end{theorem}
\begin{proof}
These estimates are obtained by applying twice (first to \(\omega^{n+1}\), then to \(u^{n+1}\)) the standard \emph{a posteriori} error estimate for the Poisson problem \cite{brenner_fem}.
\end{proof}

These bounds yield error indicators that depend only on computable residuals and not on the exact solution \( u \). In particular, if the data mismatches, namely, \( \mathbf{P}^n - \mathbf{P}^n_h \), \( \phi^n - \phi^n_h \), and \( g - g_h \), are of higher order, then the estimators given in \eqref{eq:l2erromega} and \eqref{eq:h1erru} serve as reliable indicators for the corresponding errors. These estimators can be used to locally refine the mesh \cite{ainsworth}, see \Cref{ssec:adapt}.

\begin{remark} 
The biharmonic problem \eqref{eq:secondminn} is closely related to the bending of a hinged (simply supported) plate \cite{sweers}, where the vertical deflection \(u\) satisfies
\begin{equation}\label{eq:bending}
u := \argmin_{v\in H^2(\Omega)\cap H^1_g(\Omega)} \int_\Omega \Big\{\frac{1}{2}(\Delta v)^2 - (1-\sigma)\det(D^2 v) - f v \Big\},
\end{equation}
where \(\sigma\) is the Poisson ratio and \(f\) is the applied load. When \(\sigma = 0\), the energy density reduces to \(\frac{1}{2}|D^2 v|^2\), so \eqref{eq:secondminn} is a special case of \eqref{eq:bending}. Although typical materials have \(0 < \sigma < 0.5\), the \(\sigma = 0\) case arises in certain idealized settings. The estimates in \Cref{lemma:apriori,lemma:aposteriori} also apply to \eqref{eq:bending} by replacing \(\mathrm{div}(\mathrm{div}(\mathbf{P}^n))\) with \(f\).
\end{remark}

\begin{remark}
A well-known issue, the \textit{Babu\v{s}ka paradox}, arises when modelling curved domains with polygonal approximations for \eqref{eq:secondminn} and \eqref{eq:bending}. For curvilinear boundaries, \eqref{eq:bc} reads \(\frac{d^2 g}{ds^2} = D^2 u^{n+1} : (\tau \otimes \tau) - \kappa \frac{\partial u^{n+1}}{\partial \nu} \), with \(\kappa\) the signed curvature. Replacing a smooth boundary by inscribed polygons causes the solutions to fail to converge to the true solution \cite{sweers}. Consequently, standard conforming finite elements, such as the method considered in this work, cannot be applied directly to curvilinear domains, and penalty formulations are typically introduced to enforce boundary conditions weakly; see \cite{bartels} for a recent analysis.
\end{remark}

\section{Hessian recovery}\label{sec:hessian}
In the previous section we have approximated both $u^{n+1}$ and $\omega^{n+1}=-\Delta u^{n+1}$ by piecewise linear finite elements. However, to solve the nonlinear subproblem \eqref{eq:firstmin} pointwise on each mesh vertex, we must also approximate the full Hessian $D^2u^{n+1}$ on each mesh vertex. In \cite{caboussat,dimitrios,dimitrios_adaptive}, the Hessian is approximated in a weak sense using piecewise linear finite elements, with homogeneous Dirichlet boundary conditions imposed on all components of the matrix field. This approach introduces significant approximation errors near the boundary due to the boundary conditions, and no convergence is observed for the error in the \( H^2 \) norm. To address these limitations, we adopt a two-step projection strategy inspired by standard gradient recovery techniques, as proposed in \cite{hessianrecovery}. This method provides a more accurate reconstruction of the Hessian, particularly near the boundary, and enables improved convergence properties.

First, we compute a post‑processed gradient $G_hu_h^{n+1}\in V_h(\Omega)$, i.e. a recovered gradient that achieves higher accuracy. Specifically, we employ the polynomial-preserving recovery (PPR) gradient technique introduced in \cite{nagazhang,nagazhang2}, although alternative recovery strategies could also be considered \cite{ainsworth}. We construct for each vertex $z\in\mathcal{T}_h$ a local patch $\omega_z$ of surrounding elements and fit a quadratic polynomial $p_z\in \mathbb{P}_2(\omega_z)$ in a discrete least‑squares sense to the finite‑element solution values on the vertices of $\omega_z$. To prevent deterioration of the recovered gradient near the boundary, the local patch $\omega_z$ is enlarged to include neighboring elements and their neighbors as well, as suggested in \cite{nagazhang,nagazhang2}. The recovered gradient is then defined by $(G_hu_h^{n+1})(z) := \nabla p_z(z)$, which is locally linear.  This procedure preserves all polynomials up to degree $2$ exactly. For further details, one can refer to the original work \cite{nagazhang}. 

Next, we define the recovered Hessian $D^2_hu_h^{n+1}$ by projecting the symmetrized gradient of $G_hu_h^{n+1}$ back onto the finite element space. That is, we seek $(D^2_hu_h^{n+1})_{ij}\in V_h(\Omega)$ such that 
\begin{equation}\label{eq:hessian}
\int_\Omega (D^2_hu_h^{n+1})_{ij} v_h = \dfrac{1}{2}\int_\Omega \frac{\partial (G_hu_h^{n+1})_i}{\partial x_j}v_h + \dfrac{1}{2}\int_\Omega \frac{\partial (G_hu_h^{n+1})_j}{\partial x_i}v_h,
\end{equation} 
for any $v_h\in V_h(\Omega)$ and $1\leq i,j\leq 2$. By construction, $D^2_hu_h^{n+1}$ is symmetric. 

In order to have an \textit{a priori} estimate on $\|D^2u^{n+1} - D^2_hu_h^{n+1}\|_{L^2(\Omega)}$, we 
start by defining $\Tilde{u}^{n+1}\in H^1_g(\Omega)$ as the solution to 
\begin{equation}\label{eq:untilde}
\int_\Omega \nabla \Tilde{u}^{n+1} \cdot \nabla v = \int_\Omega \omega^{n+1}_h v\quad \forall v\in H^1_0(\Omega).
\end{equation}
The following result holds.
\begin{theorem}\label{lemma:hessianrec} 
Let $\Omega\subset\mathbb{R}^2$ be a convex polygon and let $u\in H^3(\Omega)$, and $g,\frac{d^2g}{ds^2} \in H^{3/2}(\partial\Omega)$. Moreover, let assume that $G_h u_h^{n+1}$ superconverges to $\nabla \Tilde{u}^{n+1}$, i.e. there exist $C>0$ and $0<\alpha\leq 1$ independent of $h$ such that 
\begin{equation}\label{eq:hpthe}
\frac{1}{h}\|\nabla \Tilde{u}^{n+1} -G_hu_h^{n+1}\|_{L^2(\Omega)}+ \frac{1}{h^{1/2}}\|\nabla \Tilde{u}^{n+1} -G_hu_h^{n+1}\|_{L^2(\partial\Omega)}\leq Ch^\alpha,
\end{equation}
then the following estimate holds: 
\begin{equation}\label{eq:boundd2u}
\|D^2 u^{n+1} - D^2_hu_h^{n+1}\|_{L^2(\Omega)}\leq C_1h^{\alpha} + C_{2}\|\epsilon_h^{n+1}\|_{L^2(\Omega)} + \mathcal{O}(h).
\end{equation}
\end{theorem}
\begin{proof}
We observe that 
$$\|D^2 u^{n+1} - D^2_hu_h^{n+1}\|_{L^2(\Omega)}^2\leq \|D^2 u^{n+1} - D^2 \Tilde{u}^{n+1}\|_{L^2(\Omega)}^2+\|D^2 \Tilde{u}^{n+1} - D^2_hu_h^{n+1}\|_{L^2(\Omega)}^2$$
The first term can be estimated by standard regularity results for the Poisson equation, indeed: 
$$\|D^2 u^{n+1} - D^2 \Tilde{u}^{n+1}\|_{L^2(\Omega)}\leq C_\Omega\|\omega^{n+1} - \omega^{n+1}_h\|_{L^2(\Omega)}.$$
As for the second term, $\|D^2 \Tilde{u}^{n+1} - D^2_hu_h^{n+1}\|_{L^2(\Omega)}^2$, the proof follows \cite{hessianrecovery}. 
\end{proof}

Thanks to \Cref{lemma:apriori}, we know that \( \|\epsilon_h^{n+1}\|_{L^2(\Omega)} = \mathcal{O}(h^2) \), which leads to the estimate $\|D^2 u^{n+1} - D^2_h u_h^{n+1}\|_{L^2(\Omega)} = \mathcal{O}(h^{\alpha})$.
\begin{remark}
The superconvergence assumption \eqref{eq:hpthe} is consistent with the known properties of the PPR technique on mildly structured meshes; see \cite{nagazhang,nagazhang2}. In particular, under suitable regularity assumptions on the exact solution, interior superconvergence of the recovered gradient is well established. However, the corresponding boundary contribution in \eqref{eq:hpthe} is, to the best of our knowledge, not covered by existing analytical results. The parameter $\alpha \in (0,1]$ is therefore introduced to account for the observed convergence behavior, where $\alpha = 1$ corresponds to optimal second-order superconvergence. The numerical study \cite{hessianrecovery} shows that the convergence rate corresponding to $\alpha=1$ is achieved on unstructured frontal meshes, which are used in the numerical experiments of this work.
\end{remark}

\begin{remark}
Instead of defining the discrete Hessian by \eqref{eq:hessian}, another option is the approach proposed in \cite{zhang_rec}, in which the PPR technique is applied twice to define $D^2_hu_h = G_h(G_h(u_h))$. Numerical results show that, on mildly structured meshes, the accuracy orders between the method in \cite{zhang_rec} and  the one defined by \eqref{eq:hessian} are similar.
\end{remark}

\begin{remark}
Since we seek a convex solution \( u \) to the Monge-Ampère problem, it is crucial to ensure that the recovered Hessian remains symmetric positive definite. Numerical experiments indicate that this property is naturally preserved on non adapted unstructured meshes (see \Cref{fig:mesh} in \Cref{sec:num}). However, issues may arise on adaptively refined unstructured meshes, where the irregularity in local vertex distributions can lead to non-convexity of the locally reconstructed quadratic polynomial. To address this, we incorporate a regularization term \cite{caboussat} and seek $(D^2_hu_h^{n+1})_{ij}\in V_h(\Omega)$ such that 
\begin{equation*}
\begin{split}
\int_\Omega (D^2_hu_h^{n+1})_{ij} v_h + \sum_{K\in\mathcal{T}_h}|K|\int_K \nabla (D^2_hu_h^{n+1})_{ij}\cdot \nabla v_h = &\dfrac{1}{2}\int_\Omega \frac{\partial (G_hu_h^{n+1})_i}{\partial x_j}v_h\\
& + \dfrac{1}{2}\int_\Omega \frac{\partial (G_hu_h^{n+1})_j}{\partial x_i}v_h,
\end{split}
\end{equation*} 
for any $v_h\in V_h(\Omega)$, $1\leq i,j\leq 2$.
\end{remark}

\section{Solution to the nonlinear problem \eqref{eq:firstmin}}\label{ssec:nonlinear} Problem \eqref{eq:firstmin} is a nonlinear minimization problem that can be written as
\begin{equation}\label{eq:firstmin_extended}
\mathbf{P}^n = \argmin_{\mathbf{Q}\in L^2(\Omega; \mathbb{R}^{2\times 2})}\left\{\int_\Omega \dfrac{1}{2}|\mathbf{Q}|^2- D^2u^n:\mathbf{Q},\quad\text{s.t. } \text{\rm det}\:\mathbf{Q} = f,\,\mathbf{Q}\text{ spd}\right\},\\
\end{equation}
The minimization problem can be solved pointwise. Indeed, for almost any $x\in \Omega$, $\mathbf{P}^n(x)$ is the projection of $D^2u^n(x)$ onto the subset of symmetric positive definite matrices with determinant equal to $f(x)$. Moreover, the solution is unique. There are several numerical techniques available to tackle this problem. For example, one may parametrize the matrix $\mathbf{Q}$, apply a Lagrange multiplier approach to enforce the constraints, and then use Netwon's method for the resulting unconstrained minimization problem \cite{caboussat}. In the two-dimensional case, one efficient approach is the $Qmin$ algorithm, introduced in \cite{glowinski}. We briefly describe this method below and refer the reader to the original work for more details. 

We assume that there exists $c_0>0$ such that $f(x)\geq c_0$ for almost every $x\in\Omega$. We define the normalized quantities $\overline{D^2u^n}:=D^2u^n/\sqrt{f}$ and $\overline{\mathbf{P}^n}:= \mathbf{P}^n/\sqrt{f}$. Then, \eqref{eq:firstmin} becomes equivalent to the pointwise minimization problem  
\begin{equation}\label{eq:min}
\overline{\mathbf{P}^n}(x) = \argmin_{\mathbf{Q}\in\mathbb{R}^{2\times 2}}\left\{\dfrac{1}{2}|\mathbf{Q}|^2- \overline{D^2u^n}(x):\mathbf{Q},\quad\text{s.t. } \text{\rm det}\:\mathbf{Q} = 1,\,\mathbf{Q}\text{ spd}\right\},\\
\end{equation}
with $x\in\Omega$. It can be shown \cite{glowinski} that $\overline{\mathbf{P}^n}(x)$ is a solution to \eqref{eq:min} if and only if it has the spectral decomposition 
$$\overline{\mathbf{P}^n}(x) =\mathbf{S}(x)\text{\rm diag}(p_1(x),p_2(x))\mathbf{S}^T(x),$$
where $\mathbf{S}(x)$ is an orthogonal matrix of eigenvectors matrix of $\overline{D^2u^n}(x)$ and $p(x)=(p_1(x),p_2(x))$ minimizes the reduced problem: 
\begin{equation}\label{eq:qmin}
p(x) = \argmin_{q\in \mathbb{R}^{2}}\left\{q^T q - 2b^T(x) q,\quad\text{s.t.}\,\, q_1q_2 = 1\right\},
\end{equation}
with $b(x):=\text{\rm diag}(\mathbf{S}(x)^T\overline{D^2u^n}(x)\mathbf{S}(x))$, i.e. $b(x)=(b_1(x),b_2(x))$ are the eigenvalues of $\overline{D^2u^n}(x)$. Once the reduced problem is formulated, one can apply a Lagrange multiplier argument to incorporate the quadratic constraint and then solve the resulting problem via Newton's algorithm.

\subsection*{Stability and error estimates}
In practice, \eqref{eq:qmin} is solved for each vertex of $\mathcal{T}_h$. The estimates in \Cref{lemma:apriori,lemma:aposteriori} show that the errors are determined by the norm of the projection gap $\mathbf{P}^n - \mathbf{P}^n_h$. To characterize the decay of this gap under mesh refinement, we employ a two-stage argument. First, we establish a stability bound for the minimization problem \eqref{eq:min} in the appropriate norm. Second, we invoke classical interpolation estimates to translate this stability into the optimal order of convergence with respect to $h$. The following result holds. 
\begin{theorem}\label{lemma:stability}
Let $\Omega$ be a bounded convex domain, and let $\mathbf{P}$, $\mathbf{P}^n$ be the solution to \eqref{eq:firstmin_extended} with data $D^2u$ and $D^2u^n$, respectively. If $|f|\leq K$ a.e., $K>0$, and if $D^2u,D^2u^n$ are symmetric and there exist $\delta,M>0$ such that $\text{\rm tr}(D^2u(x))>\delta$, $\text{\rm tr}(D^2u^n(x))>\delta$ and $|D^2u(x)|\leq M$, $|D^2u^n(x)|\leq M$ for any $x$, then there exists $L\geq 1$ such that
\begin{equation}\label{eq:stability}
\|\mathbf{P} - \mathbf{P}^n\|_{L^2(\Omega)}\leq L \|D^2u - D^2u^n\|_{L^2(\Omega)}.
\end{equation}
\end{theorem}
\begin{proof}
Set $x\in\Omega$ and define $\mathcal M :=\{\mathbf{Q}\in \mathbb{R}^{2\times 2},\;\det \mathbf{Q}=f(x),\;\mathbf{Q} = \mathbf{Q}^T\}$. Since $f(x)>0$ and
$\nabla(\det)\mathbf{Q}=\text{adj}(\mathbf{Q})\neq0$ on $\mathcal M$, the Implicit Function Theorem \cite{johnlee} implies that $\mathcal M$ is a $C^\infty$ embedded submanifold
of $\mathbb{S}_2 := \{\mathbf{Q}\in \mathbb{R}^{2\times 2},\;\;\mathbf{Q} = \mathbf{Q}^T\}$. If $\text{\rm tr}(\mathbf H)\neq 0 $, we can define the nearest-point projection onto $\mathcal M$ as
$$\Pi_{\mathcal{M}}(\mathbf{H}) := \argmin_{\mathbf{Q}\in \mathcal{M}}|\mathbf{Q} - \mathbf{H}|.$$
For the existence and uniqueness of the nearest-point projection, one can refer to \cite{glowinski}. Moreover, if $\textit{\rm tr}(\mathbf H)\geq \delta >0 $, then $\Pi_{\mathcal M}(\mathbf{H})\succ 0$, which corresponds to the definition \eqref{eq:firstmin_extended} when $\mathbf{H}=D^2u^n$. By \cite{leobacher}, since $\mathcal M$ is a $C^\infty$ submanifold of $\mathbb{S}_2$, the map $\Pi_{\mathcal M}:U\to\mathcal M$ is $C^\infty$ on  $U:=\{\mathbf{Q}\in\mathbb{S}_2,\;\textit{\rm tr}(\mathbf Q)\geq \delta >0 \}\supset\mathcal M$.  Moreover,
on any compact $K\subset U$ its derivative is bounded $L \;=\;\sup_{Y\in K}\|D\Pi_{\mathcal M}(Y)\| \;<\;\infty$. By the hypotheses $|D^2u(x)|,|D^2u^n(x)|\le M$ and $\text{\rm tr}(D^2u(x)),\text{\rm tr}(D^2u^n(x))\ge\delta >0$, both $D^2u(x)$ and $D^2u^n(x)$ lie in a fixed compact $K\subset U$.  Therefore
\[
|\mathbf{P}(x) - \mathbf{P}^n(x)|
=|\Pi_{\mathcal M}(D^2u(x))-\Pi_{\mathcal M}(D^2u^n(x))|
\le
L|D^2u(x) - D^2u^n(x)|,
\]
which is the claimed estimate. The Lipschitz constant $L$ is bigger or equal than one, indeed, if $D^2u(x),D^2u^n(x)\in\mathcal{M}$, then  
\[
|\mathbf{P}(x) - \mathbf{P}^n(x)|
=|D^2u(x) - D^2u^n(x)|.
\]
To obtain the result it suffices to integrate the pointwise bound over \(\Omega\) and apply the definition of the \(L^2\)‐norm.
\end{proof}

This result quantifies the stability of the nearest-point projection with respect to perturbations in the data. 

We now turn to the discretized problem on the shape-regular mesh $\mathcal{T}_h$ introduced in \Cref{ssec:fe}. Given the discrete Hessian \(D^2_h u^n_h\) defined by \eqref{eq:hessian}, we define \(\mathbf{P}^n_h\) as the piecewise linear matrix field on $\mathcal{T}_h$ whose nodal values are the pointwise solutions to \eqref{eq:firstmin_extended} with input data $D^2_hu_h^n$. The next result quantifies the error introduced by this finite‐element discretization.
\begin{theorem}\label{lemma:nonlinear_estimate} 
Let $\Omega$ be a bounded convex domain with Lipschitz boundary and assume $D^2u^n\in H^2(\Omega;\mathbb{R}^{2\times 2})$. If $|f|\leq K$ a.e., $K>0$, and if $D^2u^n,D^2_hu^n_h$ are symmetric and there exist $\delta,M>0$ such that $\text{\rm tr}(D^2u^n(x))>\delta$, $\text{\rm tr}(D^2_hu^n_h(x))>\delta$ and $|D^2u^n(x)|\leq M$, $|D^2_hu^n_h(x)|\leq M$ for any $x$, then there exists $C> 0, L>1$ such that
\begin{equation}\label{eq:nonlinear_estimate}
\|\mathbf{P}^n - \mathbf{P}^n_h\|_{L^2(\Omega)}\leq Ch^2\|\mathbf{P}^n\|_{H^2(\Omega)} +L\|D^2u^n - D^2_hu^n_h\|_{L^2(\Omega)}
\end{equation}
\end{theorem}
\begin{proof}
Let $r_h:C^0(\Omega)\to V_h$ be the Lagrange interpolant \cite{brenner_fem} on $\mathcal{T}_h$ shape-regular mesh. Then, 
\begin{align*}
\|\mathbf{P}^n - \mathbf{P}^n_h\|_{L^2(\Omega)}=& \|\mathbf{P}^n - r_h(\mathbf{P}^n) + r_h(\mathbf{P}^n)- \mathbf{P}^n_h\|_{L^2(\Omega)}\\
\leq &\|\mathbf{P}^n - r_h(\mathbf{P}^n)\|_{L^2(\Omega)} + \|r_h(\mathbf{P}^n)- \mathbf{P}^n_h\|_{L^2(\Omega)}\\
\leq &Ch^2\|\mathbf{P}^n\|_{H^2(\Omega)} + CL\|D^2u^n- D^2_hu^n_h\|_{L^2(\Omega)},
\end{align*}
where we use standard interpolation estimates for $r_h$ \cite{brenner_fem} and its continuity. 
\end{proof}

If $D^2u^n = D^2_hu^n_h$, then the error converges with second-order accuracy with respect to the mesh size $h$.

\section{Error indicators for the Monge-Amp\`ere equation}\label{sec:errorind}
In \Cref{sec:biharmonic,sec:hessian,ssec:nonlinear} we have derived the error estimates for the two subproblems \eqref{eq:firstmin}-\eqref{eq:biharmonic} separately. Now, let $u$ be the solution to the least-squares problem \eqref{eq:leastsq}, and assume that we know $\mathbf{P}\in H^2(\Omega;\mathbb{R}^{2\times 2})$. Then one shows that 
\begin{align*}
\|\omega - \omega^{n+1}_h\|_{L^2(\Omega)}\lesssim&\, h^2 + \|\mathbf{P}-\mathbf{P}^n_h\|_{L^2(\Omega)} +\|\phi-\phi^n_h\|_{H^{-\frac{1}{2}}(\partial\Omega)},\\
\| \nabla (u - u^{n+1}_h)\|_{L^2(\Omega)}\lesssim&\, h + \|\omega - \omega^{n+1}_h\|_{H^{-1}(\Omega)} +\|g-g_h\|_{H^{\frac{1}{2}}(\partial\Omega)},\\
\|u -u_h^{n+1}\|_{L^2(\Omega)}\lesssim&\, h^2 + \|\omega - \omega^{n+1}_h\|_{H^{-1}(\Omega)} +\|g-g_h\|_{H^{-\frac{1}{2}}(\partial\Omega)},\\
\|D^2u -D^2_hu_h^{n+1}\|_{L^2(\Omega)}\lesssim&\,h^\alpha + \|\omega- \omega^{n+1}_h\|_{L^{2}(\Omega)},\quad 0<\alpha\leq 1.
\end{align*}
Conversely, let assume that $u\in H^4(\Omega)$ is known, then 
\begin{equation*}
\|\mathbf{P} - \mathbf{P}^{n+1}_h\|_{L^2(\Omega)}\lesssim \, h^2 + \|D^2u - D^2_hu^{n+1}_h\|_{L^2(\Omega)}.
\end{equation*}
These combined estimates identify the Hessian recovery step as the bottleneck of the iterative algorithm \eqref{eq:splitting}. In particular, even when \(\alpha=1\) (as observed for polynomial-preserving recovery (PPR) post‐processing in our numerical experiments), this term remains only first‐order in \(h\) and thus limits the overall convergence of \(\|\omega-\omega_h\|_{L^2(\Omega)}\). Indeed, compared to the estimates for the biharmonic problem alone, we expect the iterative algorithm to yield first-order convergence for the error in the $H^2$ norm. This is confirmed by the numerical results presented in \Cref{ssec:nonadapt}. The regularity assumption on the solution $u$ is standard in the error analysis of second-order fully nonlinear problems as well as of linear fourth-order problems \cite{brenner,brenner2,neilan}.

Regarding the \textit{a posteriori} bounds, \Cref{lemma:aposteriori} provides element‐wise estimators \(\eta_K,\hat\eta_K\) that control all components of the splitting error except the data perturbation (e.g. $\mathbf{P}-\mathbf{P}^{n+1}_h$).  However, in the \(H^1\)-seminorm the contribution of boundary and right-hand‐side data errors decays at the same rate, or faster, than the estimator itself.  Consequently, \(\hat\eta_K\) remains a reliable, first‐order indicator of the total error in the $H^1$ norm.  We therefore define the global refinement indicator
\[
\hat{\eta} := \left( \sum_{K \in \mathcal{T}_h} \hat{\eta}_K^2 \right)^{1/2},
\]  
where each \(\hat\eta_K\) is given in \eqref{eq:etahatK}.  As \(h\to0\), \(\hat\eta\) converges at order \(\mathcal O(h)\) in the \(H^1\)–seminorm and thus this indicator is used to adaptively refine the mesh.  

\section{Numerical results}\label{sec:num}
  \begin{figure}[tbp]
    \centering
    \includegraphics[width=0.35\linewidth]{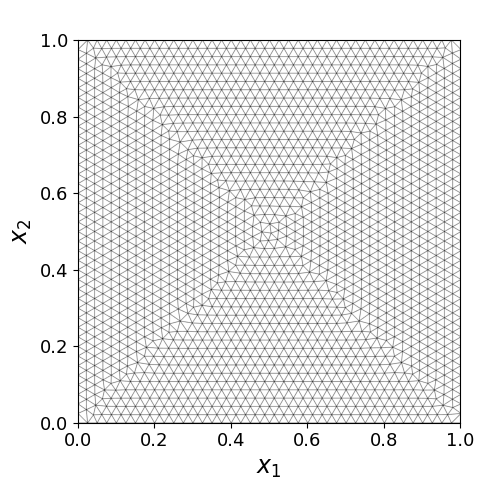}
    \caption{Unstructured frontal mesh ($h = 0.025$) generated with \texttt{bl2d}.}
    \label{fig:mesh}
  \end{figure}
We begin by validating our estimates on an independent biharmonic problem only, as in \Cref{sec:biharmonic}. Then, we validate the full framework of \Cref{sec:biharmonic,sec:hessian} on several test cases for the Monge-Amp\`ere equation, and we examine whether the \emph{a priori} and \emph{a posteriori} convergence rates from \eqref{eq:biharmonic} (see \Cref{lemma:apriori,lemma:aposteriori}) extend to the iterative algorithm. Four experiments are performed: two within the regularity assumptions, with $u\in C^{\infty}(\Omega)$, and two cases are testing robustness, with $u\notin H^2(\Omega)$. We also assess adaptive refinement driven by the estimator in \Cref{lemma:aposteriori} in \Cref{ssec:adapt}. All meshes are generated with \texttt{bl2d} \cite{bl2d}; \Cref{fig:mesh} shows a typical pre-adaptation mesh. Throughout all the experiments, the nonlinear solver for \eqref{eq:firstmin} uses the $Qmin$ algorithm from \Cref{ssec:nonlinear}, converging in $3$–$5$ iterations.

\subsection{Preliminary test case: biharmonic problem}\label{sec:numbiharmonic}
Let $\Omega=[0,1]^2$. We consider the following problem: 
\begin{equation*}
\begin{cases}
\Delta ^2u = ( x_1^4 + x_2^4 + 2x_1^2 x_2^2 + 8x_1^2 + 8x_2^2 + 8)e^{ \frac{1}{2}x_1^2 + \frac{1}{2}x_2^2 }\quad &\text{in }\Omega,\\
\Delta u = ( x_1^2 + x_2^2 + 2)e^{ \frac{1}{2}x_1^2 + \frac{1}{2}x_2^2 }\quad &\text{on }\partial\Omega,\\
u = e^{ \frac{1}{2}x_1^2 + \frac{1}{2}x_2^2 }\quad &\text{on }\partial\Omega.\\
\end{cases}
\end{equation*}
where $u_{ex}(x_1,x_2) = e^{ \frac{1}{2}x_1^2 + \frac{1}{2}x_2^2 }$. \Cref{fig:biharmonic} (left) displays the approximated solution $u_h$ with $h= 0.025$, while \Cref{fig:biharmonic} (right) shows the convergence rates of $u_h$ and $\omega_h$ and its derivatives as $h \to 0$. We confirm the expected rates predicted for \eqref{eq:biharmonic}; namely:
$$\|u - u_h\|_{L^2(\Omega)} = \mathcal{O}(h^2), \quad \|\nabla(u - u_h)\|_{L^2(\Omega)} = \mathcal{O}(h), $$
and 
$$\|\omega - \omega_h\|_{L^2(\Omega)} = \mathcal{O}(h^2),\quad \|\nabla(\omega - \omega_h)\|_{L^2(\Omega)} = \mathcal{O}(h).$$
Finally, \Cref{fig:biharmonic2dd} shows the comparison between the Hessian recovery strategy proposed in \eqref{eq:hessian} and the method proposed in \cite{zhang_rec}. In both cases, first-order accuracy is obtained for the approximation of the Hessian matrix. Therefore, in the remaining test cases the method \eqref{eq:hessian} is used.
\begin{figure}[tbp]
\centering
\begin{subfigure}{0.4\textwidth}
\centering
\includegraphics[width=0.8\textwidth]{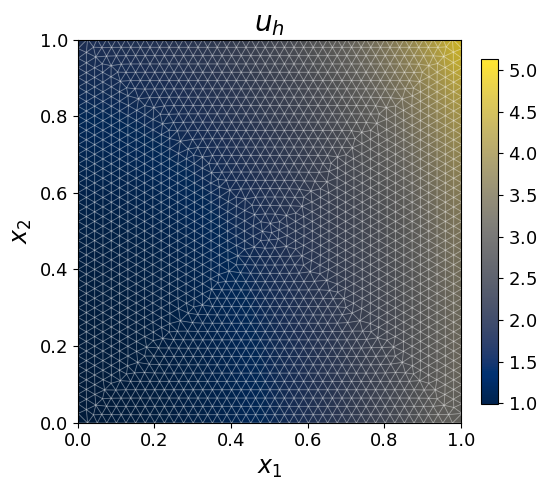}
\end{subfigure}
\hspace{0,2cm}
\begin{subfigure}{0.45\textwidth}
\centering
\includegraphics[width=0.8\textwidth]{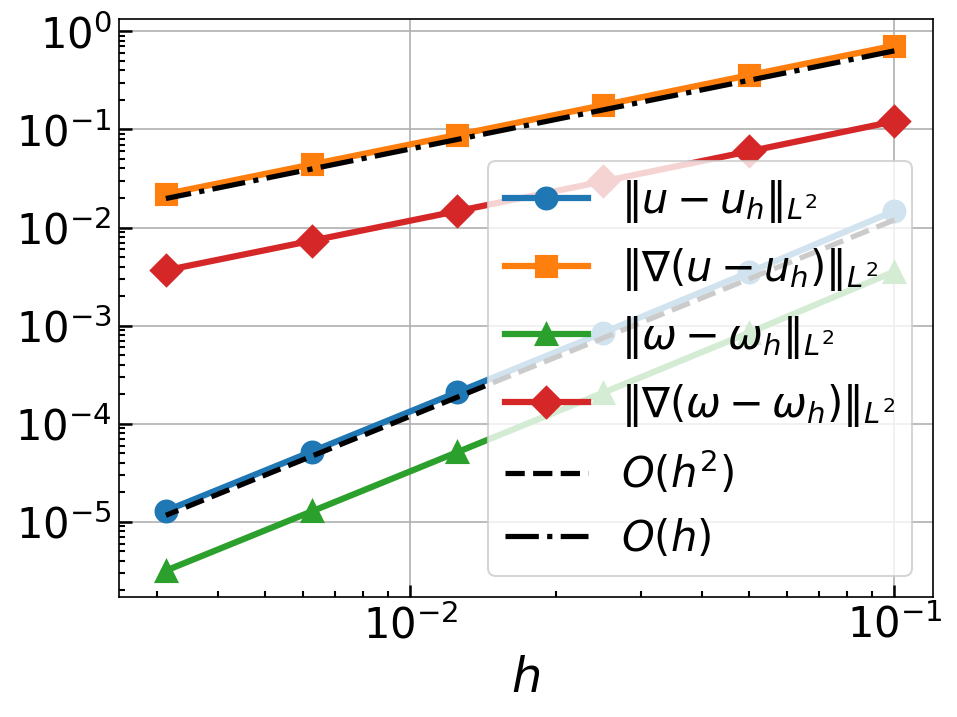}
\end{subfigure}
\caption{Biharmonic test problem. Left: plot of the numerical solution $u_h$ ($h = 0.025$). Right: errors vs. $h$.}
\label{fig:biharmonic}
\end{figure}

\begin{figure}[tbp]
\centering
\includegraphics[width=0.35\textwidth]{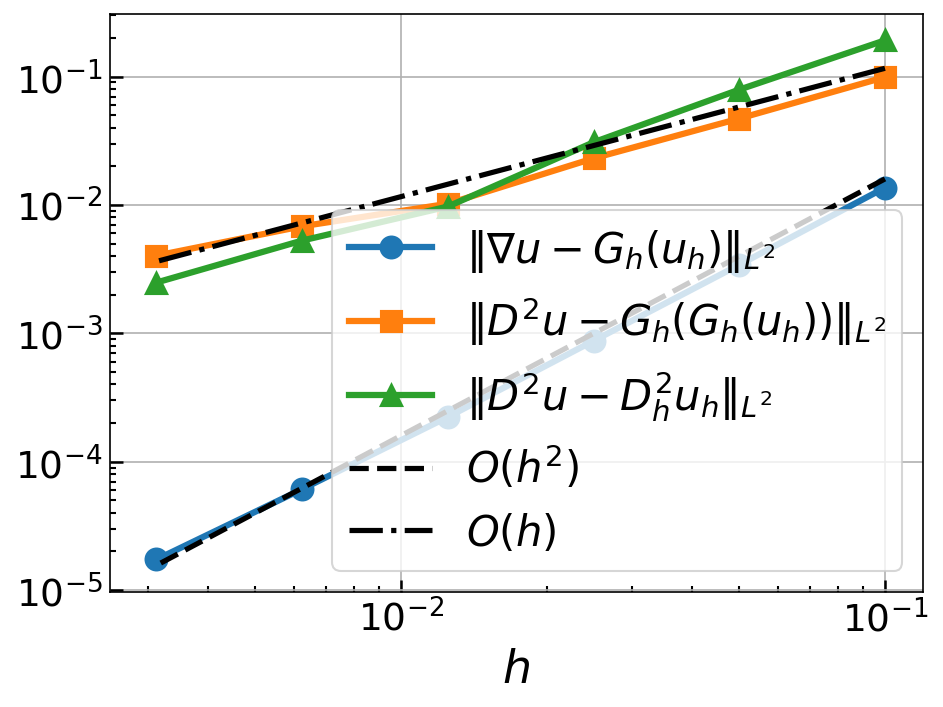}
\caption{Biharmonic test problem.  Comparison of discrete Hessian techniques: errors vs. $h$. The approximation $D^2_hu_h$ is obtained with \eqref{eq:hessian} while the approximation $G_h(G_h(u_h))$ is the approximation obtained with the method proposed in \cite{zhang_rec}.}
\label{fig:biharmonic2dd}
\end{figure}
  
\subsection{Numerical results on non-adapted meshes}\label{ssec:nonadapt}

\subsubsection{First test case} \label{sssec:example1}
Let $\Omega = [0,1]^2$, and consider the test problem defined by  
\[
f(x_1, x_2) = (1 + x_1^2 + x_2^2) e^{x_1^2 + x_2^2}, \quad   g(x_1, x_2) = e^{\frac{1}{2}(x_1^2 + x_2^2)},
\]
whose exact solution is the smooth radial function  $u(x_1, x_2) = e^{\frac{1}{2}(x_1^2 + x_2^2)}$, $(x_1, x_2) \in \Omega$. \Cref{fig:exp_approx} (left) displays the approximated solution $u_h$, while \Cref{fig:exp_approx} (right) shows the  pointwise error. In \Cref{fig:exp_iter} (left), we plot the decay of the error in $H^2$ norm as the number of splitting iterations increases. The number of iterations required for convergence grows as the mesh is refined, reaching approximately $25$ iterations for the smallest mesh size ($h = 0.00625$). A similar convergence trend is observed for $\|D^2_h u_h^{n} - \mathbf{P}^n_h\|_{L^2(\Omega)}$, consistent with the discussion in \Cref{rem:distance} (see \Cref{fig:exp_iter} (right)). \Cref{fig:exp_h} (left) presents the convergence rates of $u_h$ and its derivatives as $h \to 0$. It confirms the expected rates discussed in \Cref{sec:errorind}; namely:
\[
\|u - u_h\|_{L^2(\Omega)} = \mathcal{O}(h^2), \quad
\|\nabla(u - u_h)\|_{L^2(\Omega)} = \mathcal{O}(h), \quad
\|\omega - \omega_h\|_{L^2(\Omega)} = \mathcal{O}(h).
\]
These results implies that the error $\|\omega -\omega_h\|_{H^{-1}(\Omega)}$ scales at least as $h^2$ for this numerical example. The results are confirmed in \Cref{tab:errormin}. Furthermore, due to the improved accuracy of the post-processed gradient $G_h$, which converges with order $\mathcal{O}(h^2)$, the overall error in $H^2$ norm also exhibits linear convergence with respect to $h$; $\|D^2_h u_h^n - \mathbf{P}^n_h\|_{L^2(\Omega)}$ itself decays approximately linearly in $h$, making it a reliable proxy for the error in $H^2$ norm (\Cref{fig:exp_h} (right)). Finally, in \Cref{fig:cputimes} we compare the CPU times of the method proposed in \cite{caboussat}, in which the biharmonic problem is solved by an additional iterative procedure based on the conjugate gradient method in Sobolev spaces, with those of the direct method proposed in this work.  For the conjugate gradient algorithm, the residual stopping criterion is chosen as in \cite{caboussat} with tolerance $10^{-8}$ and maximum number of iterations of $70$. The method we propose achieves total CPU times that are roughly two orders of magnitude smaller than those of the method in \cite{caboussat}.

\begin{figure}[tbp]
\centering
\begin{subfigure}[t]{0.45\linewidth}
\centering
\includegraphics[width=0.84\linewidth]{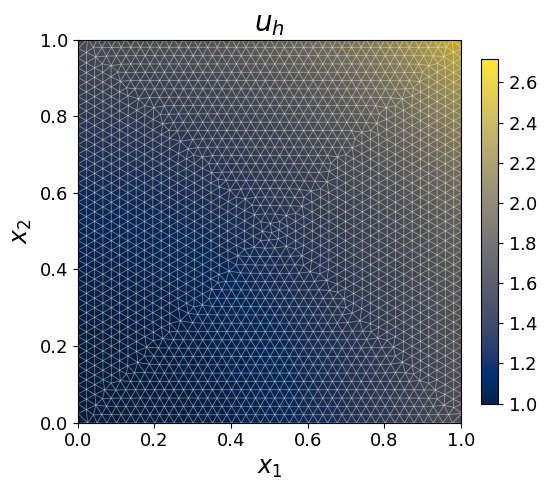}
\end{subfigure}
\begin{subfigure}[t]{0.45\linewidth}
\centering
\includegraphics[width=0.84\linewidth]{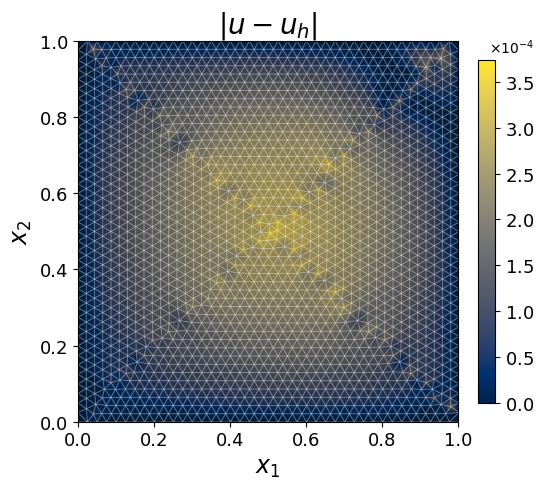}
\end{subfigure}
\caption{First test problem. Left: plot of the numerical solution $u_h$ ($h = 0.025$). Right: plot of the pointwise error ($h = 0.025$).}
\label{fig:exp_approx}
\end{figure}

\begin{figure}[tbp]
\centering
\begin{subfigure}[t]{0.42\linewidth}
\centering
\includegraphics[width=0.85\linewidth]{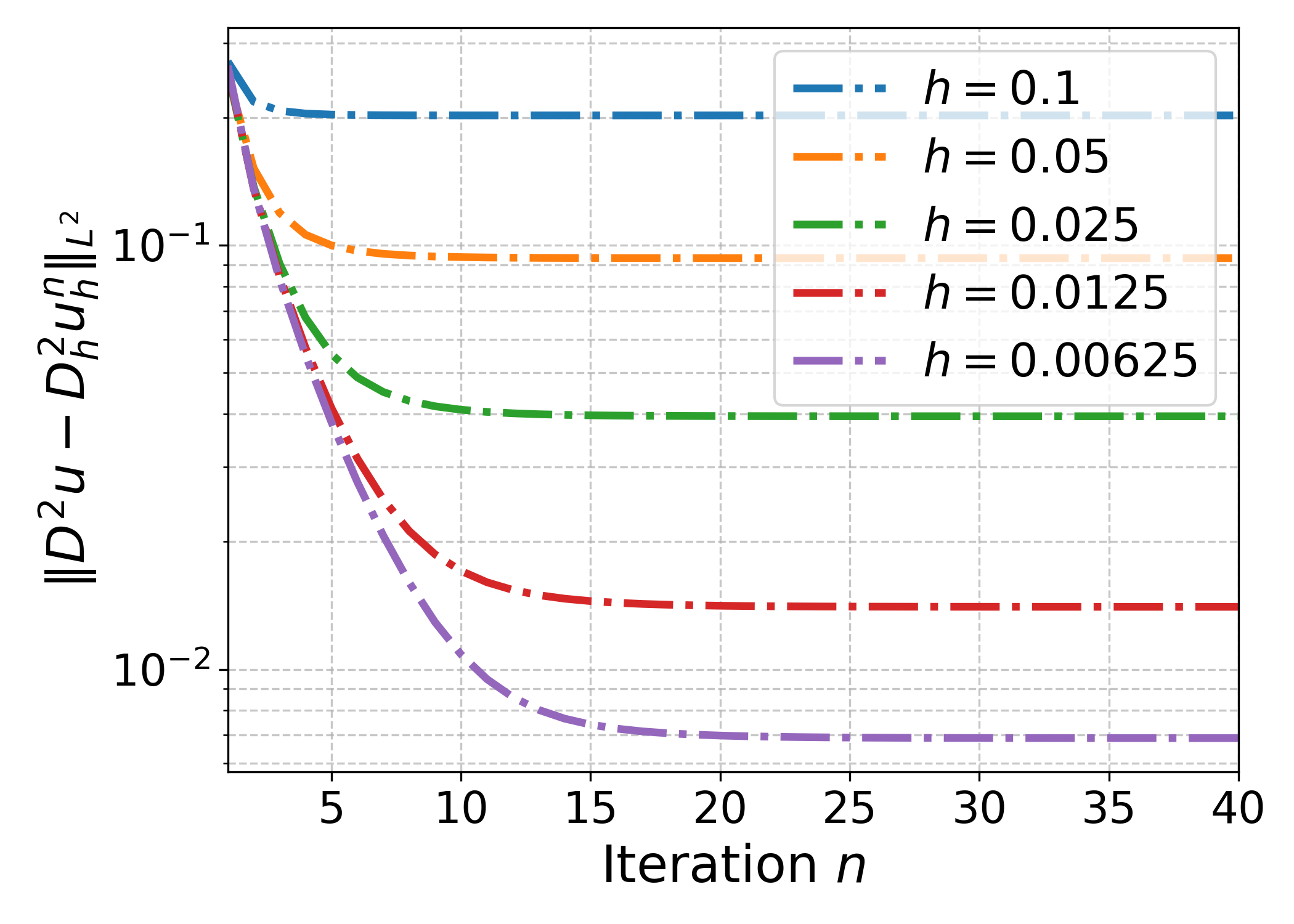}
\end{subfigure}
\hspace{0.3cm}
\begin{subfigure}[t]{0.42\linewidth}
\centering
\includegraphics[width=0.85\linewidth]{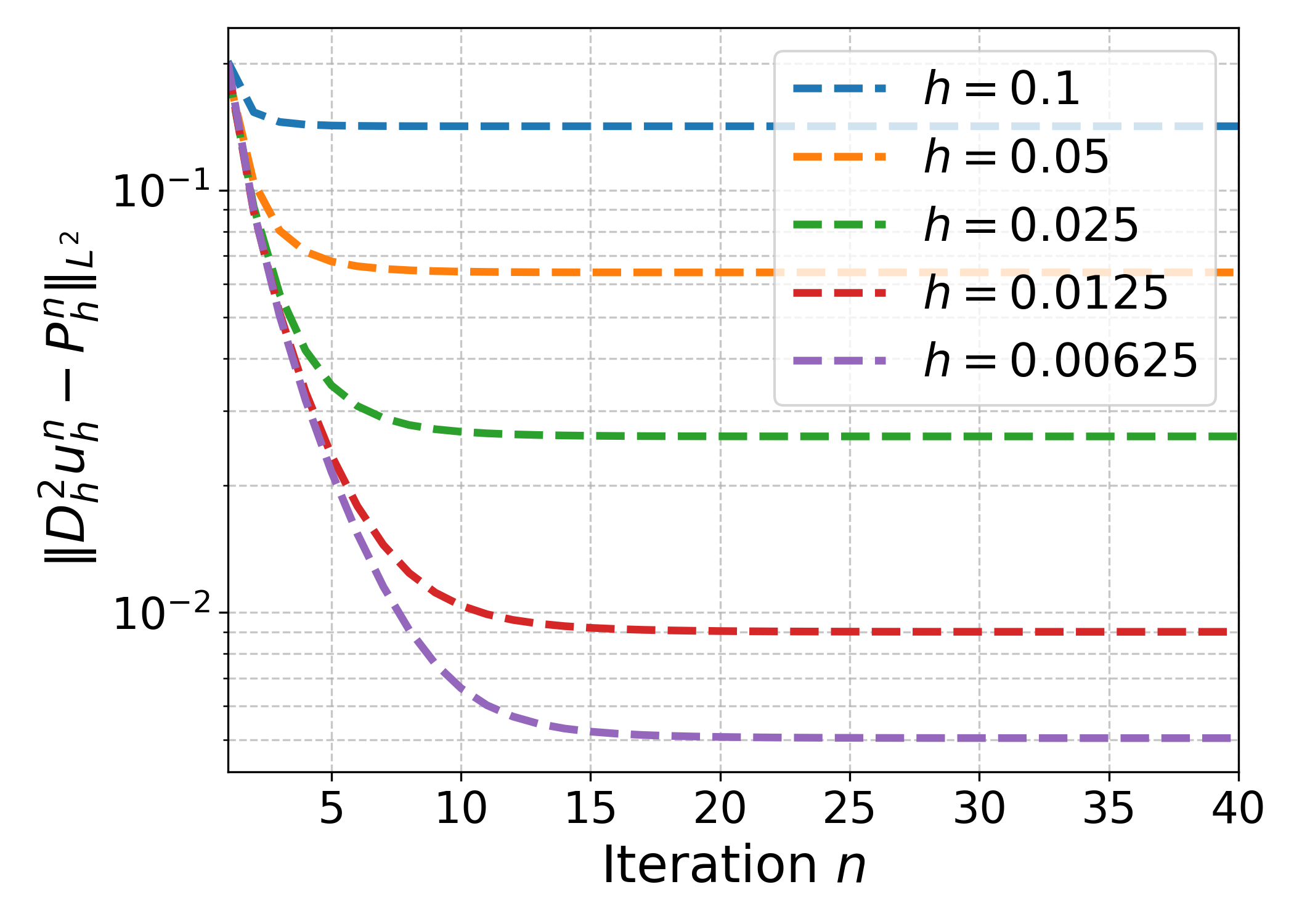}
\end{subfigure}
\caption{First test problem. Left: $\|D^2u^n - D^2_hu_h^n\|_{L^2(\Omega)}$ vs. splitting iterations for different values of $h$. Right: $\|D^2_hu_h^n - \mathbf{P}^n_h\|_{L^2(\Omega)}$ vs. splitting iterations for different values of $h$.}
\label{fig:exp_iter}
\end{figure}

\begin{figure}[tbp]
\centering
\begin{subfigure}[t]{0.42\linewidth}
\centering
\includegraphics[width=0.85\linewidth]{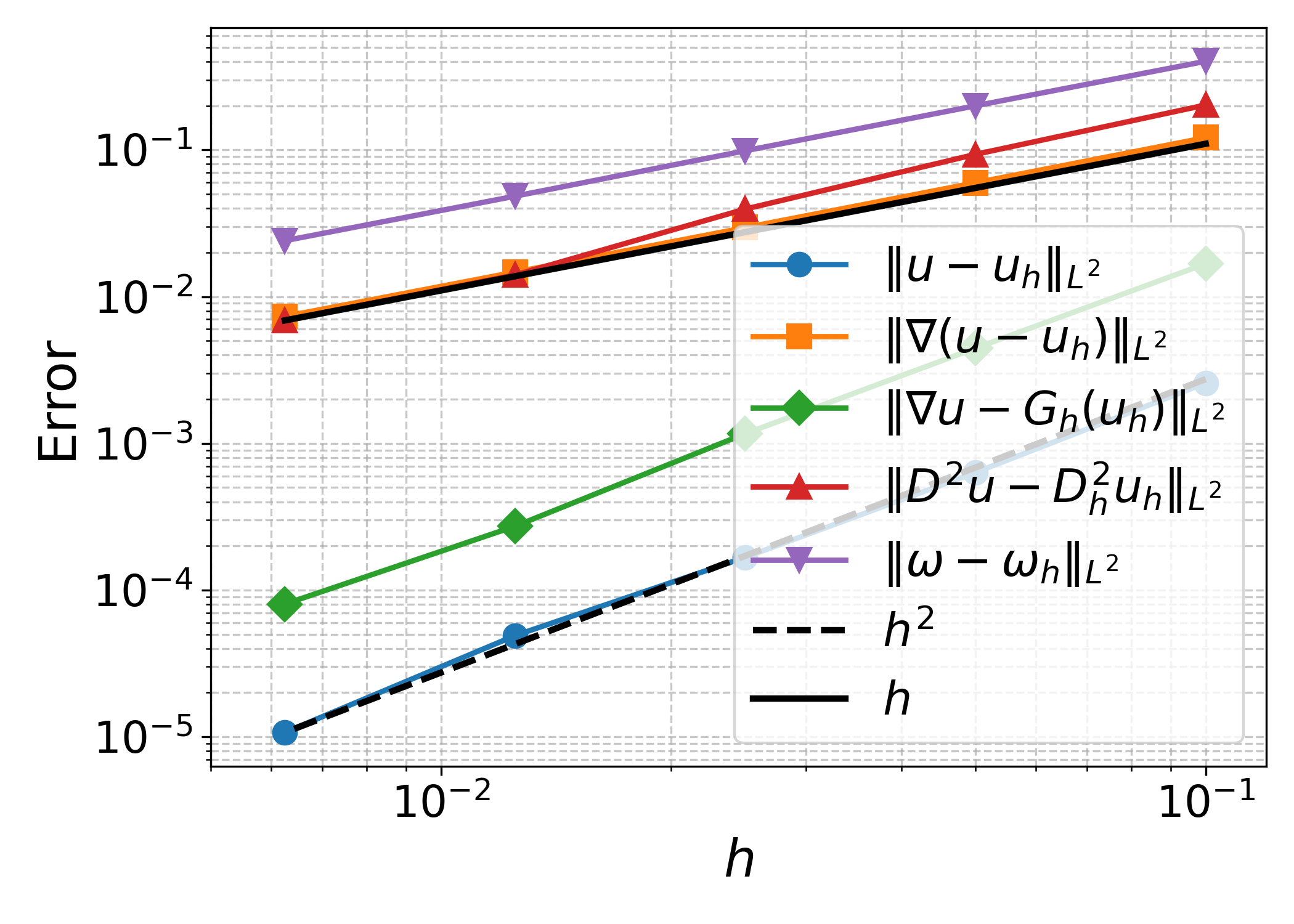}
\end{subfigure}
\hspace{0.3cm}
\begin{subfigure}[t]{0.42\linewidth}
\centering
\includegraphics[width=0.85\linewidth]{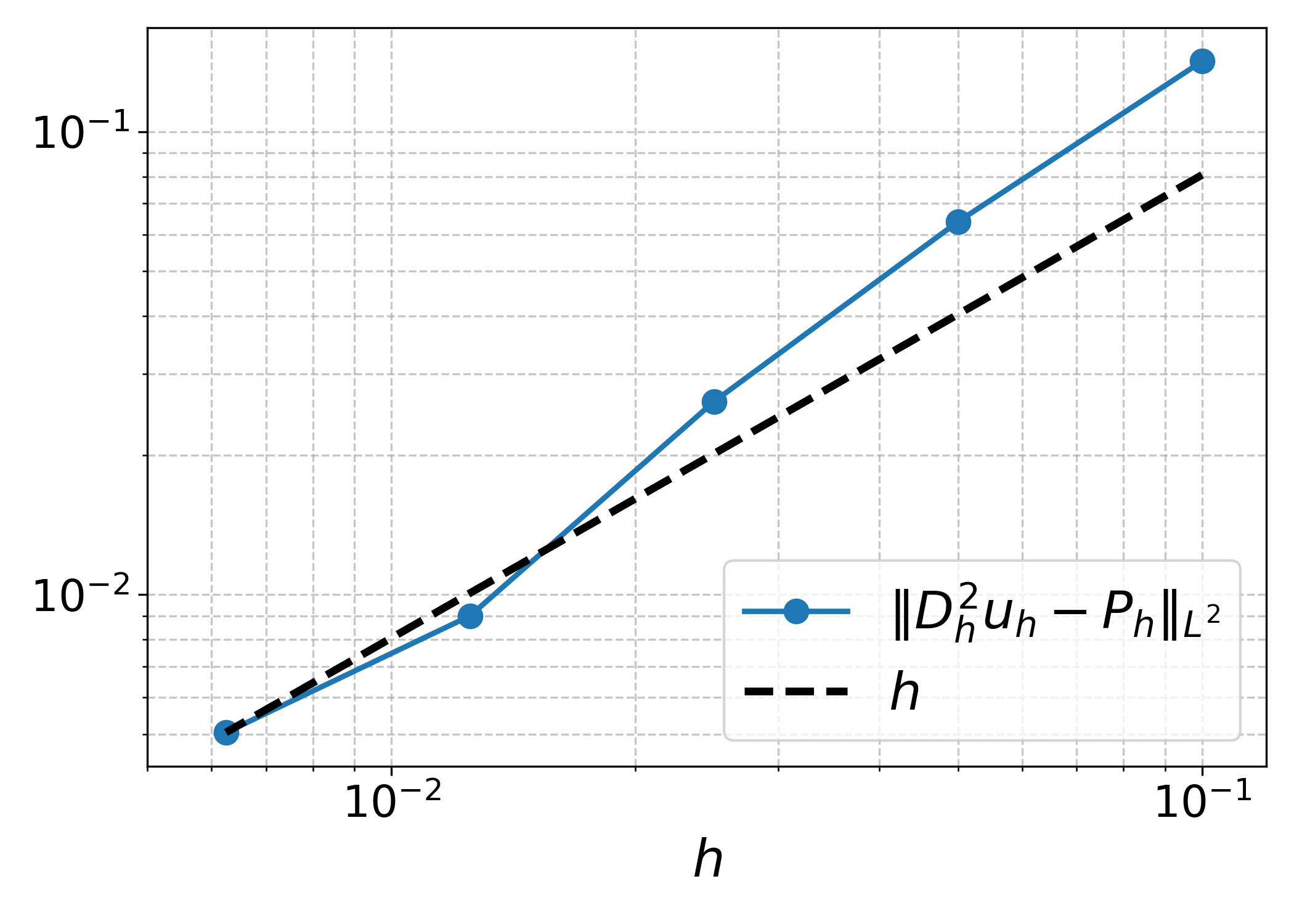}
\end{subfigure}
\caption{First test problem. Left: errors vs. $h$. Right: $\|D^2_hu_h - \mathbf{P}_h\|_{L^2(\Omega)}$ vs. $h$.}
\label{fig:exp_h}
\end{figure}

\begin{figure}[tbp]
\centering
\includegraphics[width=0.6\linewidth]{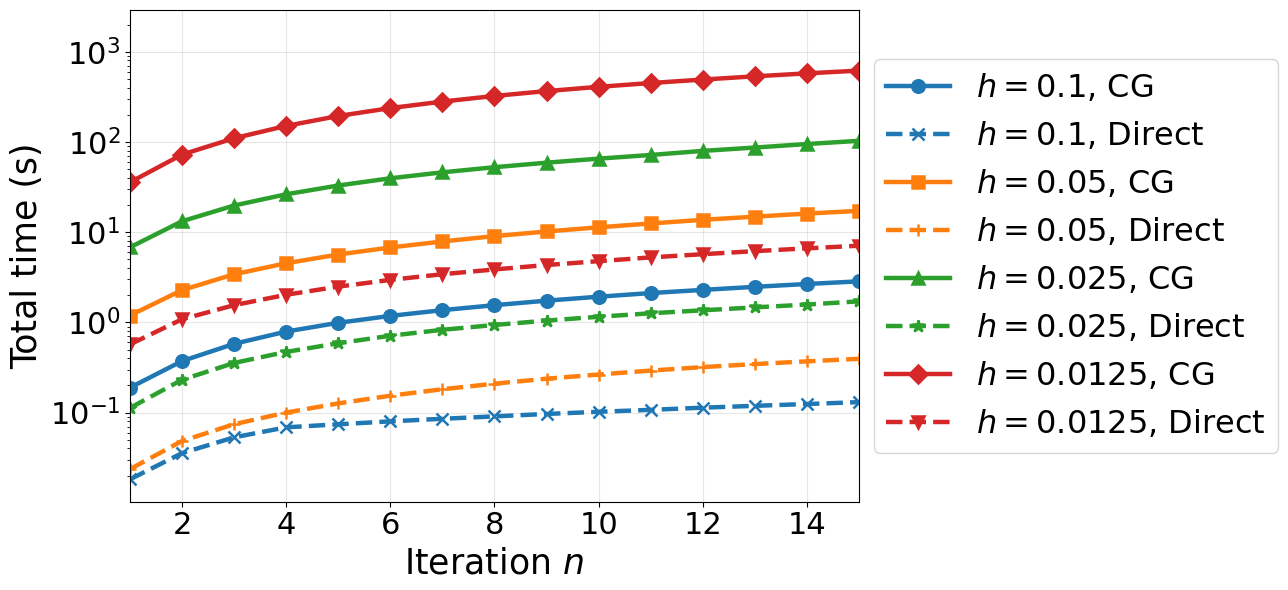}
\caption{First test problem. CPU times (s) vs. splitting iterations for different values of $h$. The continuous lines corresponds to the CG-based method proposed in \cite{caboussat} while the dashed lines corresponds to the direct method introduced in this work.}
\label{fig:cputimes}
\end{figure}

\begin{table}[tbp]
\caption{Error $\|\omega-\omega_h\|_{H^{-1}(\Omega)}$ for the first and second test cases.}
\label{tab:errormin}
\centering
\begin{tabular}{c| c c c c c } 
 $h$ & $0.1$ & $0.05$ &$0.025$ & $0.0125$ &$0.00625$ \\ 
 \hline
 \hline
 First test case & $0.6807$ & $0.1514$ & $0.0491$ & $0.0125$ & $0.0026$ \\ 
 \hline
Third test case, $R=2$ & $7.6\cdot 10^{5}$ & $2.5\cdot 10^{-5} $&$6.3\cdot 10^{-6} $&$2.3\cdot 10^{-6} $&$4.4\cdot 10^{-7}$ \\ 
\end{tabular}
\end{table}

\subsubsection{Second test case}\label{sssec:example3}

Let $\Omega = [0,1]^2$ and consider the test problem defined, for \( R \geq \sqrt{2} \), by
\[
f(x_1, x_2) = \frac{R^2}{\left(R^2 - (x_1^2 + x_2^2)\right)^2}, \quad 
g(x_1, x_2) = -\sqrt{R^2 - (x_1^2 + x_2^2)},
\]
whose exact solution is the convex function $u(x_1, x_2) = -\sqrt{R^2 - (x_1^2 + x_2^2)}$, with $(x_1, x_2) \in \Omega$. When \( R > \sqrt{2} \), the exact solution \( u \) belongs to \( C^\infty(\overline{\Omega}) \). However, when \( R = \sqrt{2} \), $u$ is smooth on every compact subset of $\Omega$ but \( u \notin H^2(\Omega) \), due to the singularity of the gradient of $u$ at the corner \((1,1)\). This makes it particularly interesting to investigate the performance of the algorithm and the quality of the approximation as \( R \to \sqrt{2}^+ \). To this end, we consider three representative values: \( R = 2 \), \( R = \sqrt{2} + 0.1 \), and \( R = \sqrt{2} + 0.01 \). Notably, for the smallest value of \( R \), convergence could not be achieved in the original work of \cite{caboussat}. \Cref{fig:root_approx} displays the graphs of the computed solutions \( u_h \) for each value of \( R \) with mesh size \( h = 0.025 \), while \Cref{fig:root_error} shows the corresponding nodal errors. As \( R \) decreases, the error becomes more concentrated near the singularity at \((1,1)\). Nevertheless, in contrast to the findings in \cite{caboussat}, our method achieves convergence even for \( R = \sqrt{2} + 0.01 \), as evidenced in \Cref{fig:root_herr}. The observed convergence orders are consistent with those predicted and $\|\omega -\omega_h\|_{H^{-1}(\Omega)}$ scales at least as $h^2$ for this numerical example (see \Cref{tab:errormin}). Moreover the number of splitting iterations to reach convergence is around $20$ for the smallest mesh size independently of the value of $R$. Lastly, we consider the critical case $\sqrt{2}$. Here, neither the \textit{a priori} nor the \textit{a posteriori} estimates from \Cref{lemma:apriori,lemma:aposteriori} apply, yet it remains useful to evaluate how our algorithm performs when the exact solution fails to meet the regularity requirements of the least‐squares formulation \eqref{eq:leastsq}. \Cref{fig:root_sqr2} (left) plots the discretization errors against the mesh size $h$. Even in this singular setting, the error in $L^2$ norm converges at a rate $\mathcal{O}(h^{3/2})$, while the error in the $H^2$ norm decays like $\mathcal{O}(h^{1/2})$. The asymptotic rates are further confirmed by the error as function of the splitting iteration $n$ (\Cref{fig:root_sqr2} (right)).
\begin{figure}[tbp]
\centering
\begin{subfigure}[t]{0.32\linewidth}
\centering
\includegraphics[width=0.97\linewidth]{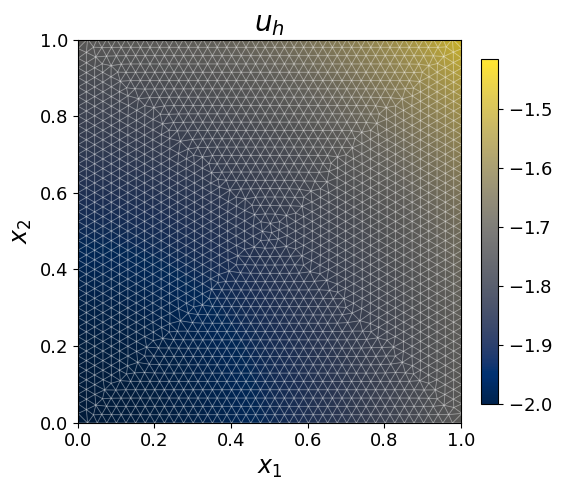}
\caption{$R=2$.}
\end{subfigure}
\begin{subfigure}[t]{0.32\linewidth}
\centering
\includegraphics[width=0.97\linewidth]{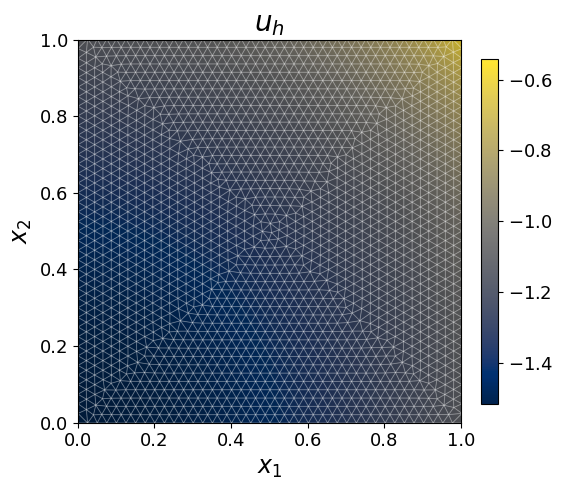}
\caption{$R=\sqrt{2} + 0.1$.}
\end{subfigure}
\begin{subfigure}[t]{0.32\linewidth}
\centering
\includegraphics[width=0.97\linewidth]{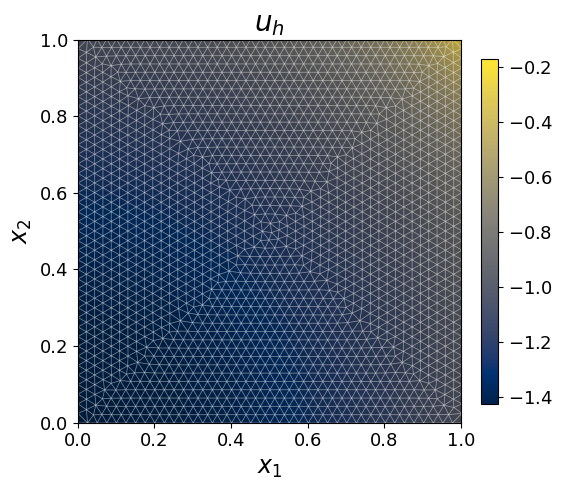}
\caption{$R=\sqrt{2} + 0.01$.}
\end{subfigure}
\caption{Second test problem. Plots of the numerical solution $u_h$ ($h = 0.025$) for $R=\{2,\sqrt{2}+0.1,\sqrt{2}+0.01\}$.}
\label{fig:root_approx}
\end{figure}

\begin{figure}[tbp]
\centering
\begin{subfigure}[t]{0.32\linewidth}
\centering
\includegraphics[width=0.97\linewidth]{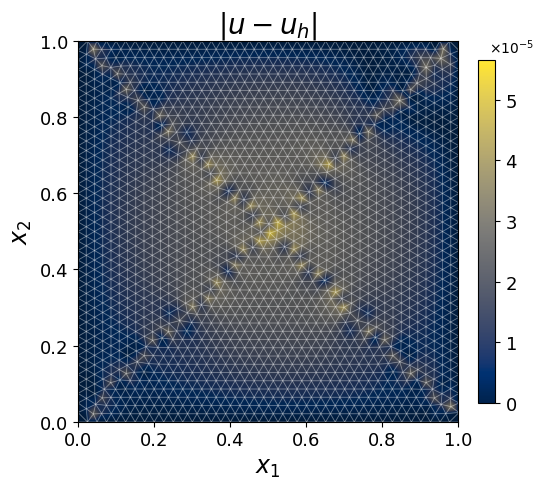}
\caption{$R=2$.}
\end{subfigure}
\begin{subfigure}[t]{0.32\linewidth}
\centering
\includegraphics[width=0.97\linewidth]{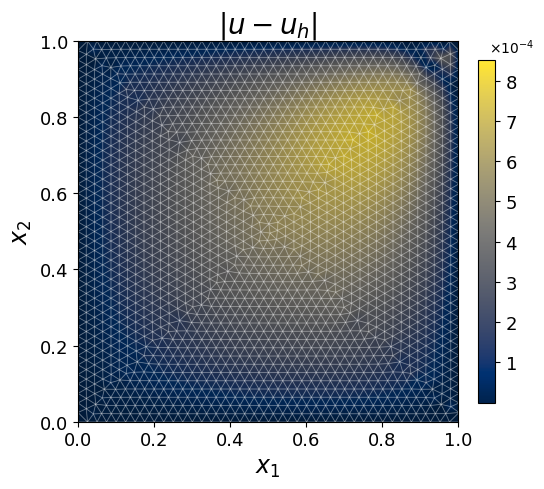}
\caption{$R=\sqrt{2} + 0.1$.}
\end{subfigure}
\begin{subfigure}[t]{0.32\linewidth}
\centering
\includegraphics[width=0.97\linewidth]{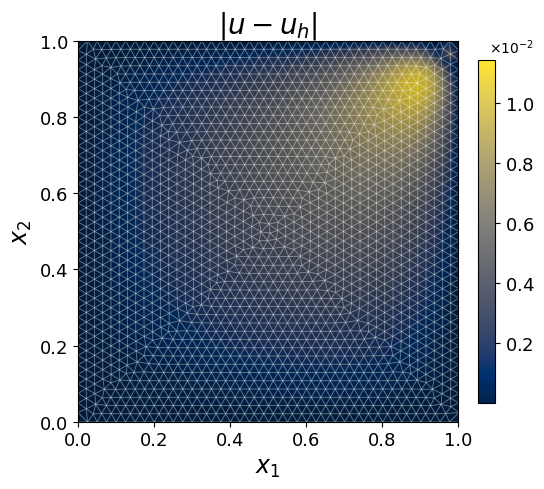}
\caption{$R=\sqrt{2} + 0.01$.}
\end{subfigure}
\caption{Second test problem. Plots of the pointwise error ($h = 0.025$) for $R=\{2,\sqrt{2}+0.1,\sqrt{2}+0.01\}$.}
\label{fig:root_error}
\end{figure}

\begin{figure}[tbp]
\centering
\begin{subfigure}[t]{0.325\linewidth}
\centering
\includegraphics[width=0.98\linewidth]{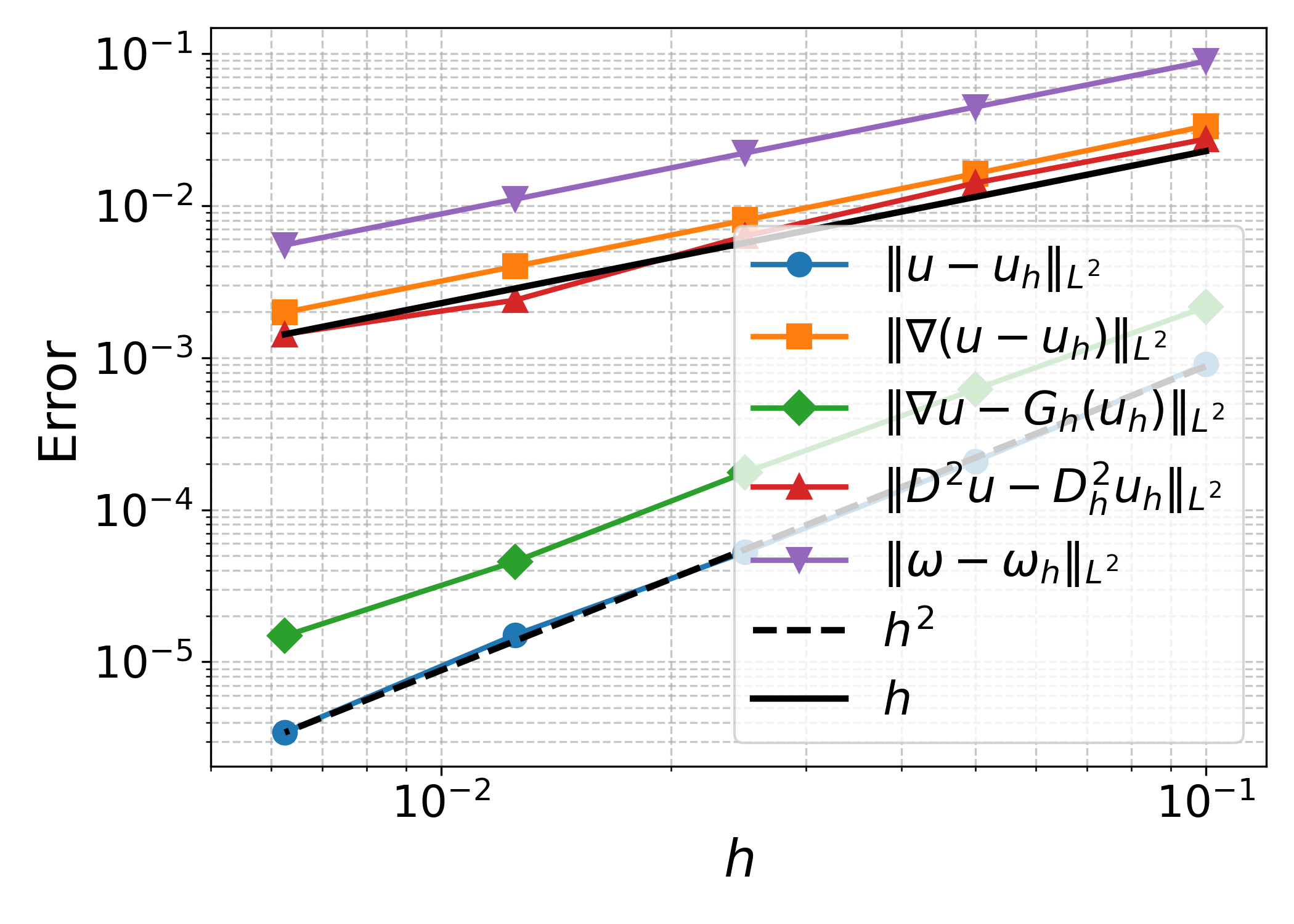}
\caption{$R=2$.}
\end{subfigure}
\begin{subfigure}[t]{0.325\linewidth}
\centering
\includegraphics[width=0.98\linewidth]{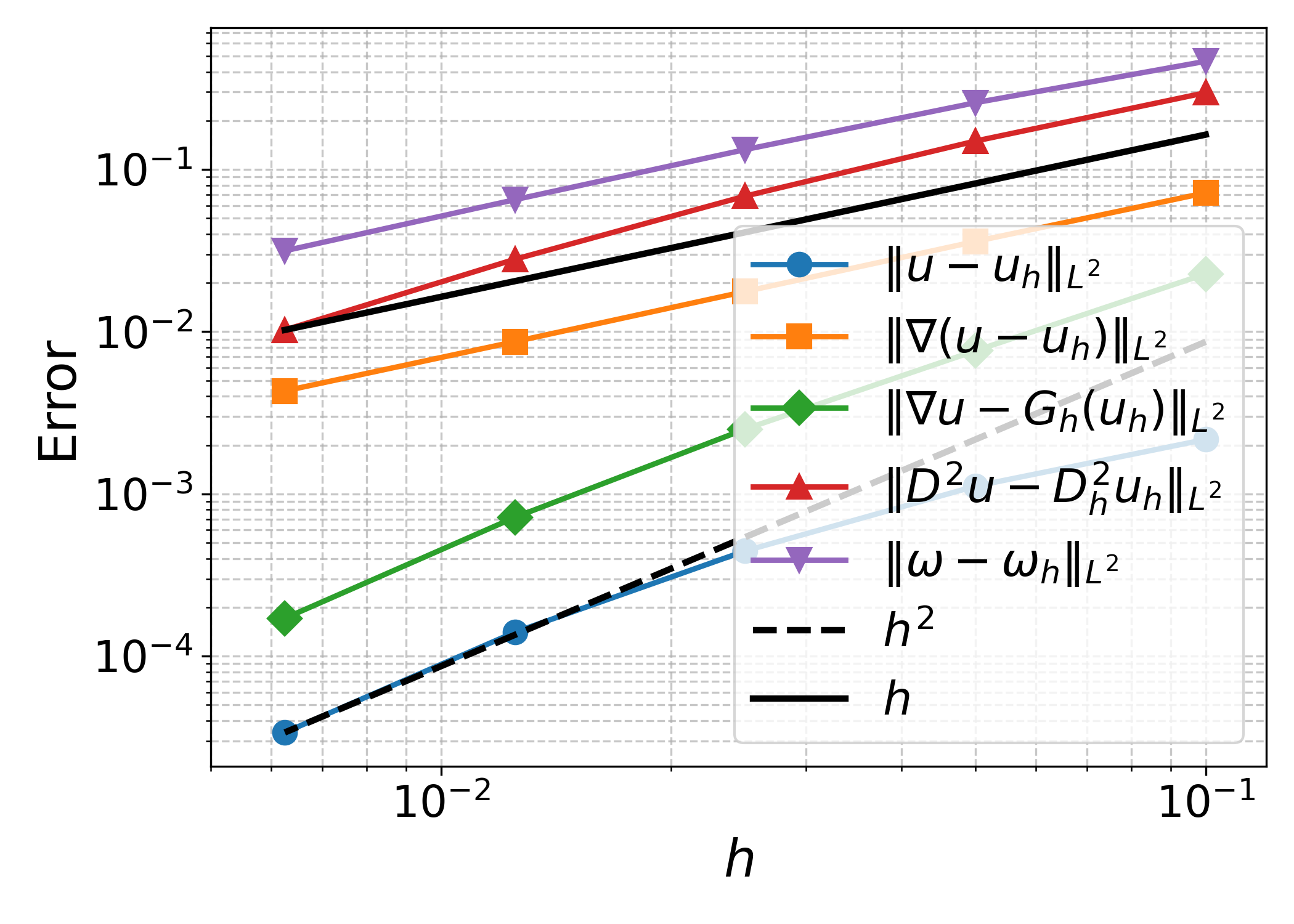}
\caption{$R=\sqrt{2} + 0.1$.}
\end{subfigure}
\begin{subfigure}[t]{0.325\linewidth}
\centering
\includegraphics[width=0.98\linewidth]{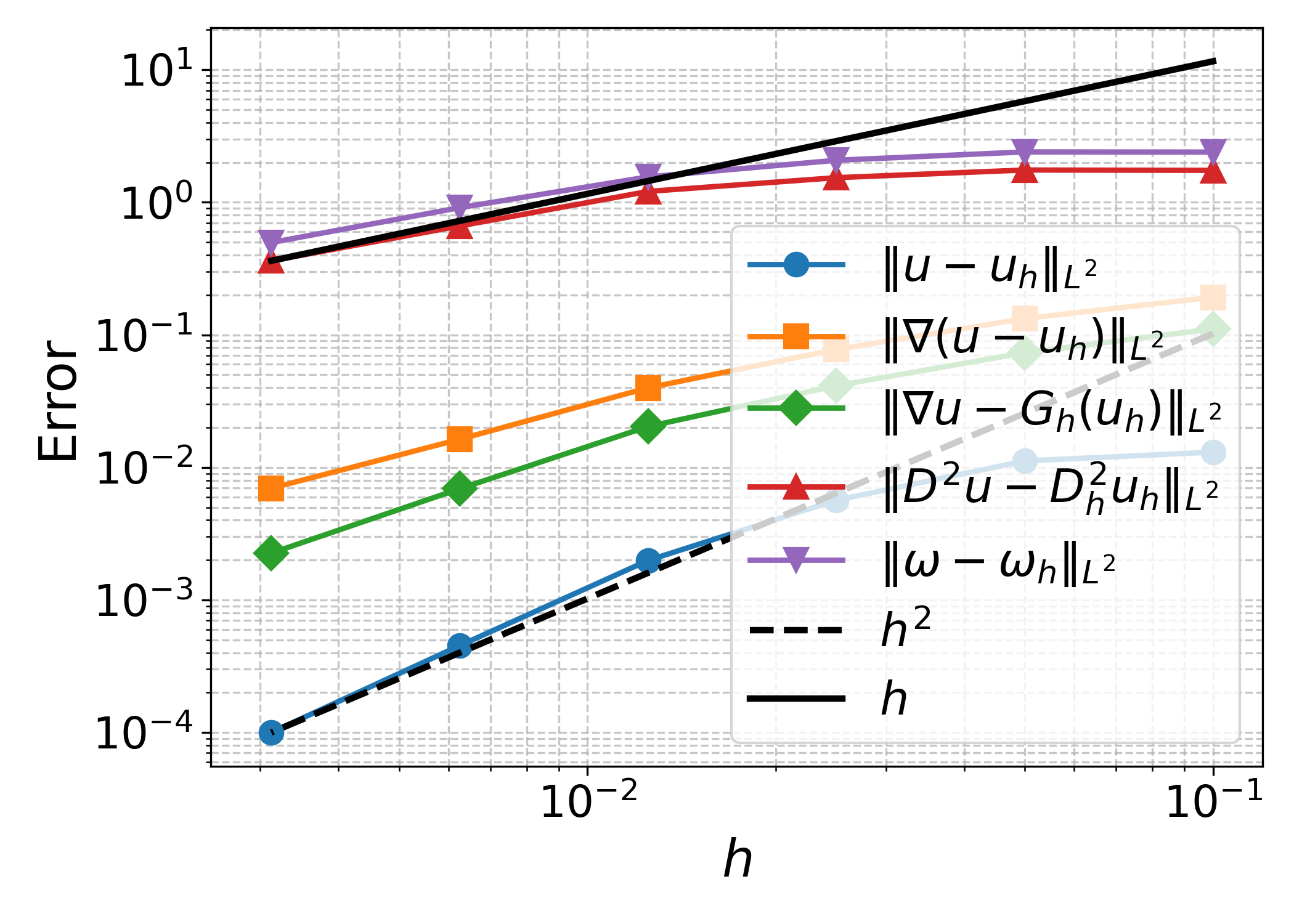}
\caption{$R=\sqrt{2} + 0.01$.}
\end{subfigure}
\caption{Second test problem. Error vs. $h$.}
\label{fig:root_herr}
\end{figure}

\begin{figure}[tbp]
\centering
\begin{subfigure}[t]{0.42\linewidth}
\centering
\includegraphics[width=0.8\linewidth]{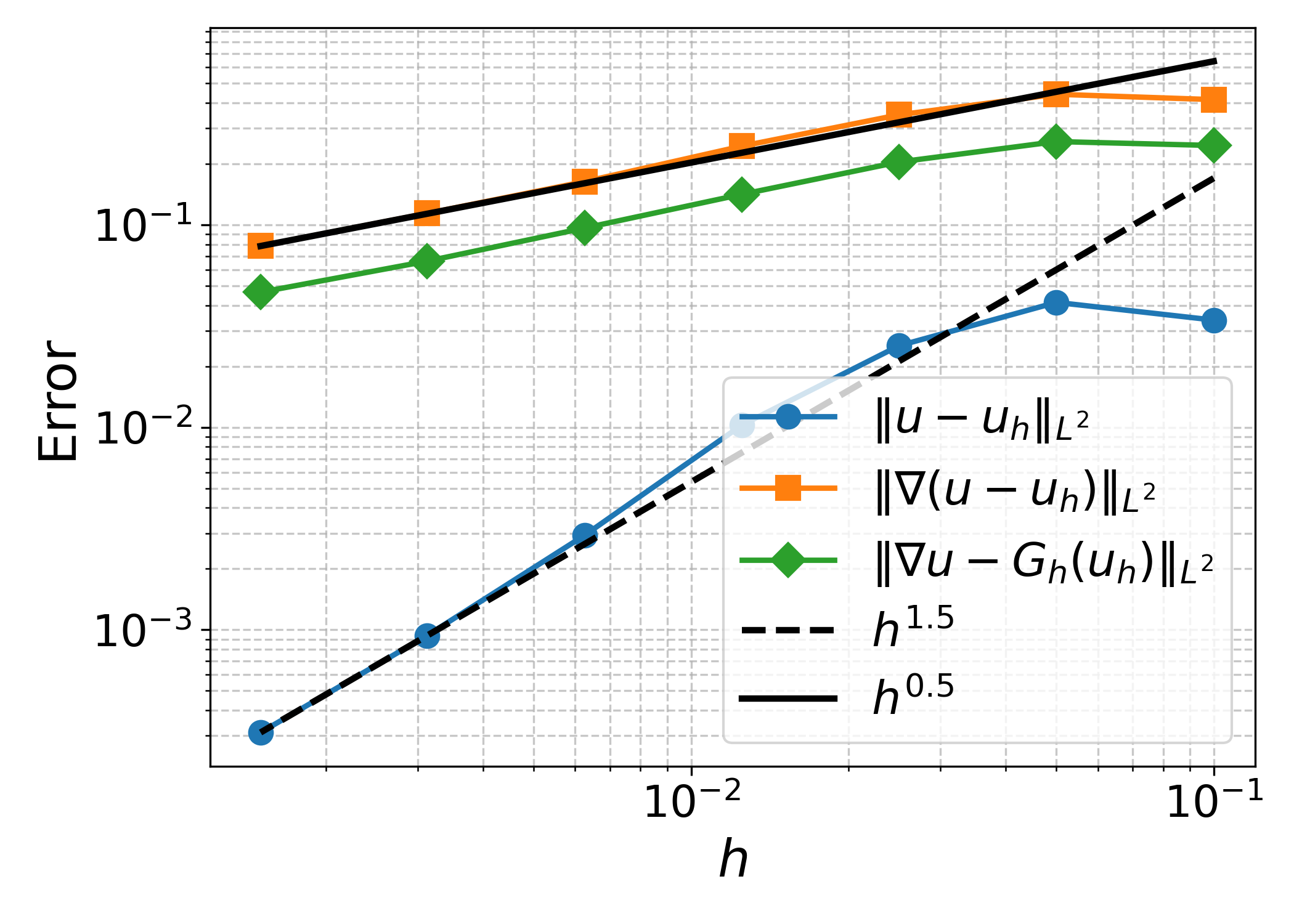}
\end{subfigure}
\hspace{0.3cm}
\begin{subfigure}[t]{0.42\linewidth}
\centering
\includegraphics[width=0.8\linewidth]{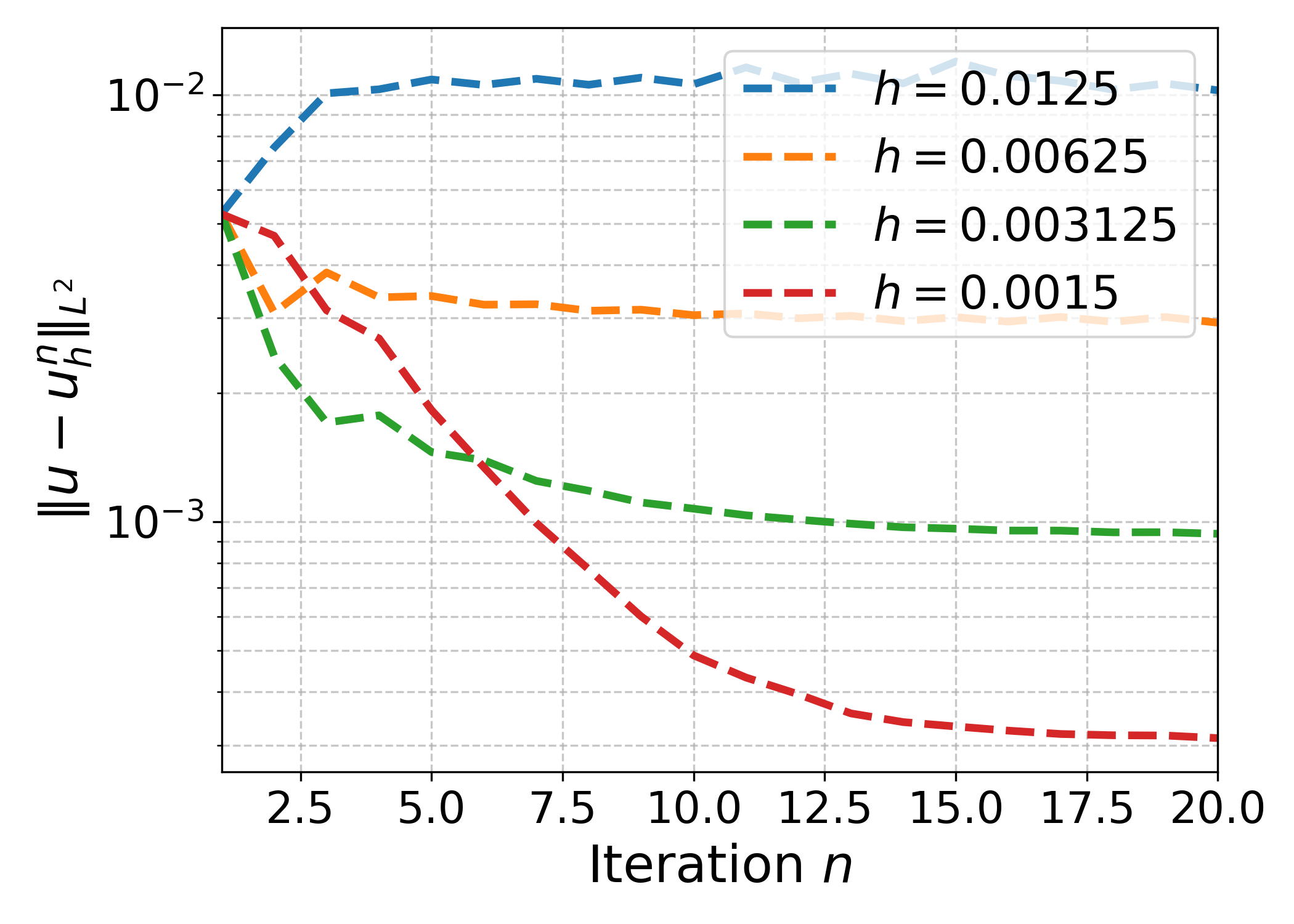}
\end{subfigure}
\caption{Second test problem, $R=\sqrt{2}$. Left: errors vs. $h$. Right: $\|u -u_h^n\|_{L^2(\Omega)}$ vs. splitting iterations.}
\label{fig:root_sqr2}
\end{figure}

\subsubsection{Third test case}
We consider another nonsmooth example. Let $\Omega=[0,1]^2$, the problem is defined by $f(x_1, x_2) = 1$ and $g(x_1, x_2) = 0$. In this case, the Monge-Amp\`ere equation does not have solutions belonging to $H^2(\Omega)$ (it does, however, admit \textit{so-called} viscosity solutions \cite{gutierrez}), despite the smoothness of the data. The issue stems from the non-strict convexity of $\Omega$ \cite{gutierrez} and indeed the lack of regularity of the solution $u$ concentrates around the corners. Therefore, the solution obtained can only be compared with computational results from the literature, e.g. \cite{neilan,lagrangian,caboussat,dimitrios_adaptive}. \Cref{fig:f1_approx} illustrates the approximated solution $u_h$ as well as $\text{\rm det}\:(D^2_hu_h)$. From the latter plot, it is clear that the numerical method fails to approximate the solution close to the corners. In order to have a better grasp of it, we also show some cross-section of the approximated solution $u_h$ (\Cref{fig:f1_values}). In particular, we observe that along the line $x_1=x_2$ (left), the approximated solution looses its convexity close to the boundary (i.e. close to the corners). However, as expected, the solution reaches its minimum in the middle of $\Omega$. As $h$ decreases, the minimum decreases and the magnitude aligns with other numerical results from the literature, e.g. \cite{neilan,caboussat}. On the other hand, we observe that as $h\to 0$, the determinant across the line approaches $1$ from below (\Cref{fig:f1_values}, right). Finally, \Cref{fig:f1_iter} shows $\|D^2_hu_h^n - \mathbf{P}^n_h\|_{L^2(\Omega)}$ as function of $h$ (left) and splitting iteration $n$ (right). The quantity decays when $h\to 0$ and $n$ increases. However, around $100$ iterations are needed to reach convergence for the smallest choiche of $h$. This slow convergence was observed also in \cite{caboussat}.
\begin{figure}[tbp]
\centering
\begin{subfigure}[t]{0.45\linewidth}
\centering
\includegraphics[width=0.84\linewidth]{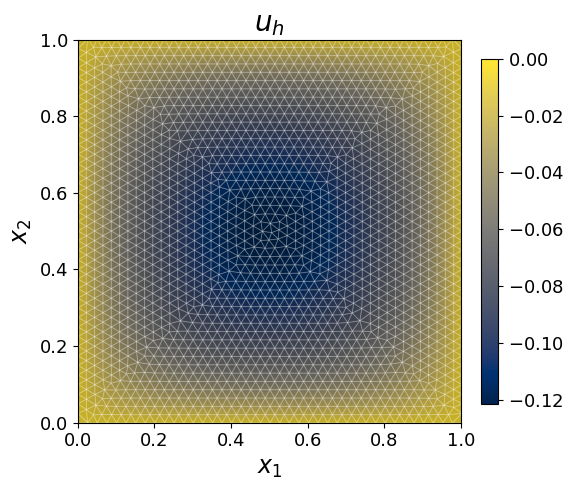}
\end{subfigure}
\begin{subfigure}[t]{0.45\linewidth}
\centering
\includegraphics[width=0.8\linewidth]{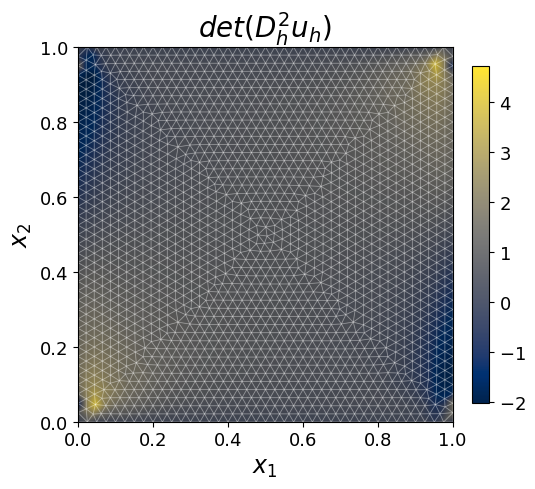}
\end{subfigure}
\caption{Third test problem. Left: plot of the numerical solution $u_h$ ($h = 0.0125$). Right: plot of $\text{\rm det}\:(D^2_hu_h)$ ($h = 0.0125$).}
\label{fig:f1_approx}
\end{figure}

\begin{figure}[tbp]
\centering
\begin{subfigure}[t]{0.42\linewidth}
\centering
\includegraphics[width=0.85\linewidth]{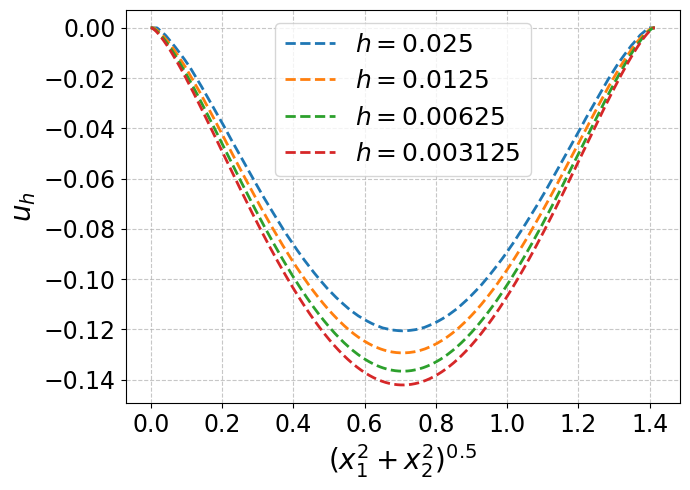}
\end{subfigure}
\hspace{0.3cm}
\begin{subfigure}[t]{0.42\linewidth}
\centering
\includegraphics[width=0.85\linewidth]{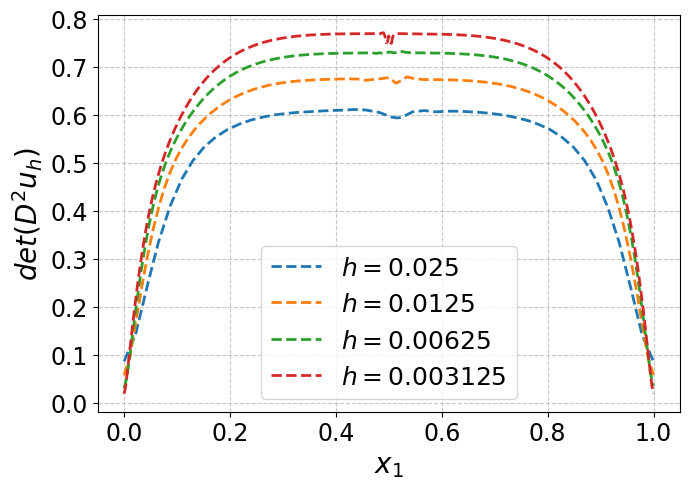}
\end{subfigure}
\caption{Third test problem. Left: plot of the numerical solution $u_h$ along the line $x_2=x_1$. Right: plot of $\text{\rm det}\:(D^2_hu_h)$ along the line $x_2=0.5$.}
\label{fig:f1_values}
\end{figure}

\begin{figure}[tbp]
\centering
\begin{subfigure}[t]{0.42\linewidth}
\centering
\includegraphics[width=0.85\linewidth]{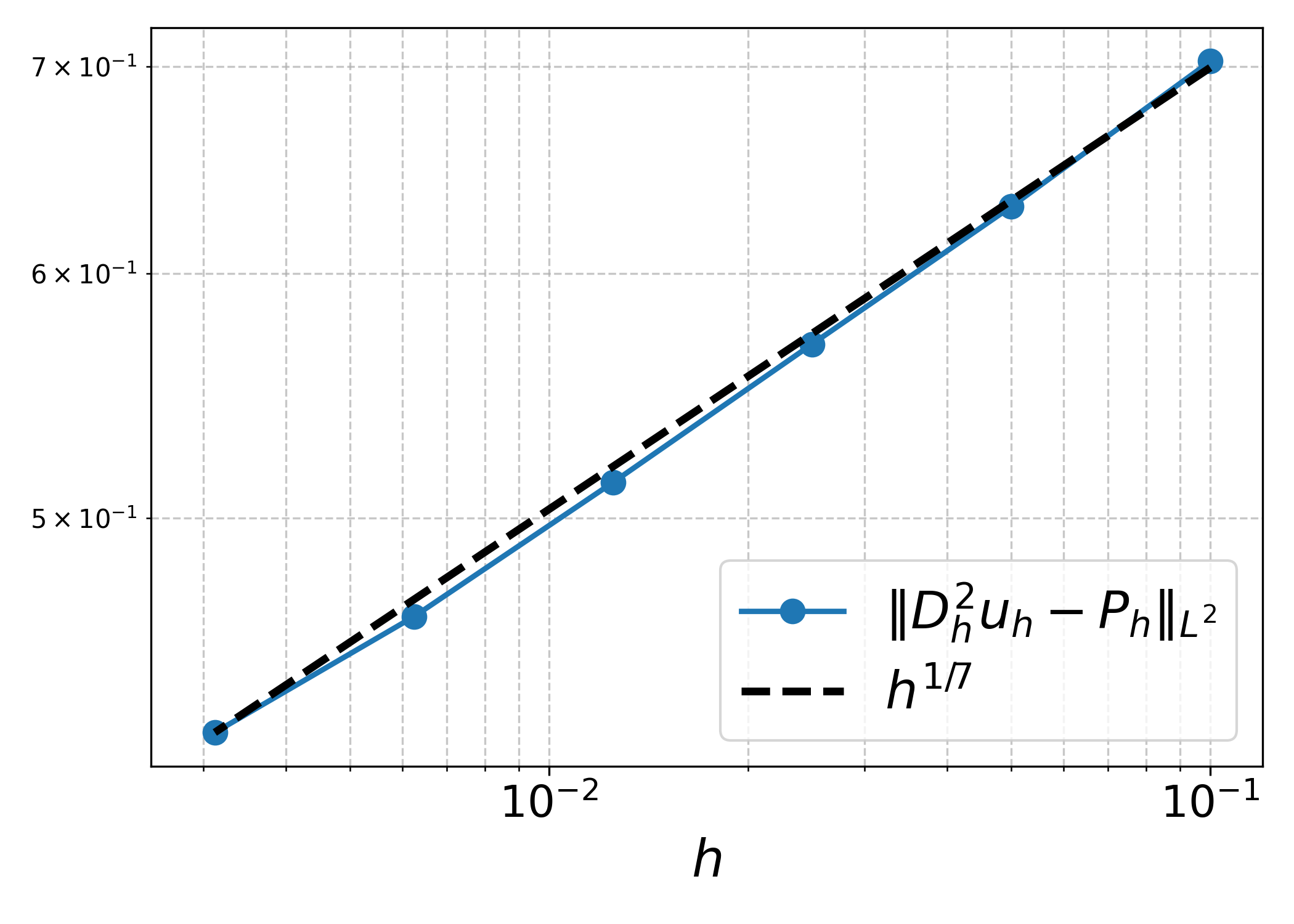}
\end{subfigure}
\hspace{0.3cm}
\begin{subfigure}[t]{0.42\linewidth}
\centering
\includegraphics[width=0.85\linewidth]{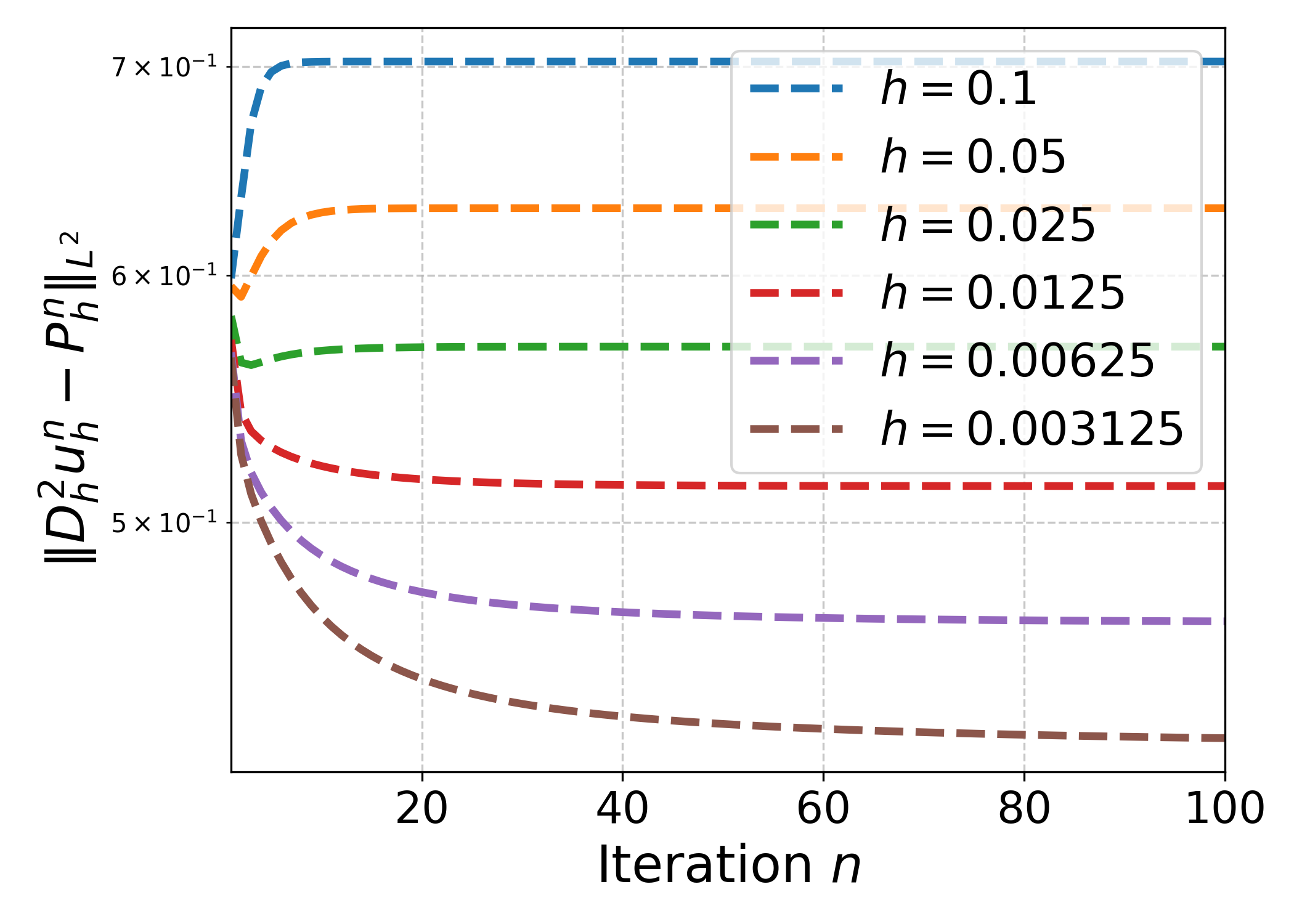}
\end{subfigure}
\caption{Third test problem. Left: $\|D^2_hu_h - \mathbf{P}_h\|_{L^2(\Omega)}$ vs. $h$. Right: $\|D^2_hu_h^n - \mathbf{P}^n_h\|_{L^2(\Omega)}$ vs. splitting iterations for different values of $h$.}
\label{fig:f1_iter}
\end{figure}

\subsubsection{Fourth test case}\label{sssec:example5}
To conclude this section with numerical experiments on non-adapted meshes, we consider a final non-smooth case. The solution of the associated problem is the convex function $u$ defined by
\[
u(x) =  \sqrt{(x_1-0.5)^2 + (x_2-0.5)^2},
\]
a function that does not possess $H^2$ regularity when $(0.5,0.5) \in \Omega$, and satisfies $M u = \pi \delta_{(0.5,0.5)}$, where $M$ denotes the Monge-Amp\`ere measure (see, e.g., \cite{dephilippis,gutierrez}) and $\delta_{(0.5,0.5)}$ is the Dirac measure at $(0.5,0.5)$. In particular we consider the Monge-Amp\`ere problem on $\Omega = [0,1]^2$, and the problem reads:
\begin{equation*}\label{eq:delta}
\begin{cases}
\text{\rm det}\: D^2 u(x_1, x_2) = \pi\delta_{(0.5,0.5)} & \text{in } \Omega, \\
u(x_1, x_2) = \sqrt{(x_1-0.5)^2 + (x_2-0.5)^2} & \text{on } \partial \Omega.
\end{cases}
\end{equation*}
In particular, the solution to this problem is unique. Since our method is suited for strictly positive right-hand sides $f$, as suggested in \cite{caboussat}, we approximate the Dirca measure with $\delta_{(0.5,0.5)} = \frac{\varepsilon^2}{\pi(\varepsilon^2 + (x_1-0.5)^2 + (x_2-0.5)^2)^2}$, where $\varepsilon > 0$ is a small positive number. \Cref{fig:root_central_approx} illustrates the computed solution $u_h$ and the pointwise error for $h=0.025$ and $\varepsilon = 10^{-2}$. As expected, the error concentrates around the point $(0.5,0.5)$. However the least-squares methodology is also able to approximate these singular problems. This is confirmed by the error convergence shown in \Cref{fig:root_central_iter}. For both the error in $L^2$ and $H^1$ norms we recover a decay order of $\mathcal{O}(h)$.
\begin{figure}[tbp]
\centering
\begin{subfigure}[t]{0.45\linewidth}
\centering
\includegraphics[width=0.84\linewidth]{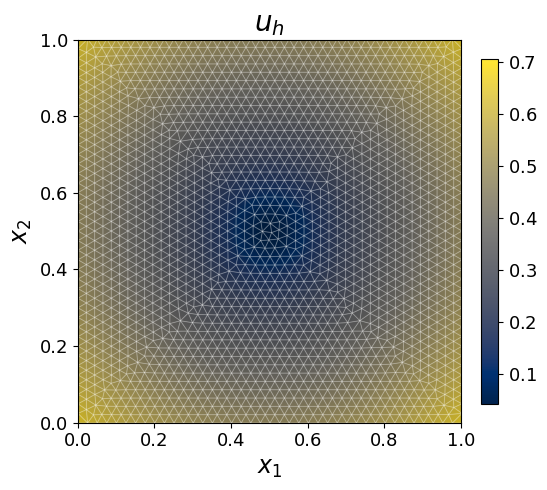}
\end{subfigure}
\begin{subfigure}[t]{0.45\linewidth}
\centering
\includegraphics[width=0.84\linewidth]{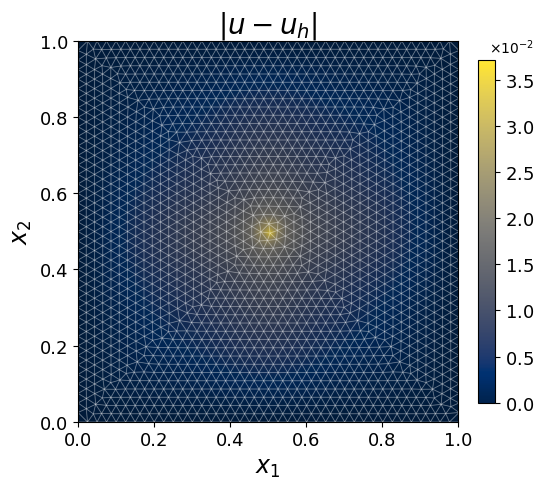}
\end{subfigure}
\caption{Fourth test problem. Left: plot of the numerical solution $u_h$ ($h = 0.025$). Right: pointwise error of the numerical solution $u_h$ ($h = 0.025$).}
\label{fig:root_central_approx}
\end{figure}

\begin{figure}[tbp]
\centering
\begin{subfigure}[t]{0.42\linewidth}
\centering
\includegraphics[width=0.85\linewidth]{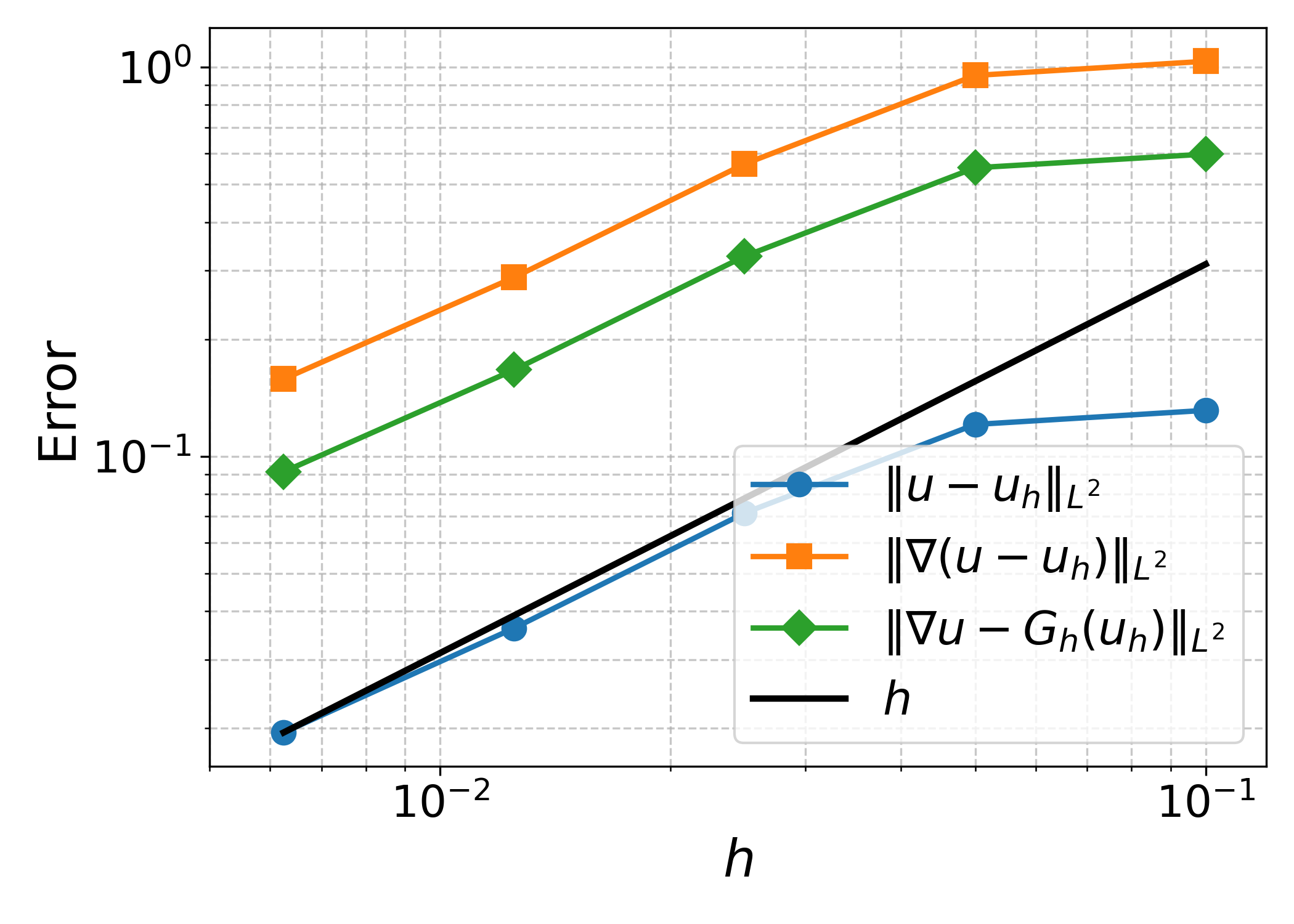}
\end{subfigure}
\hspace{0.3cm}
\begin{subfigure}[t]{0.42\linewidth}
\centering
\includegraphics[width=0.85\linewidth]{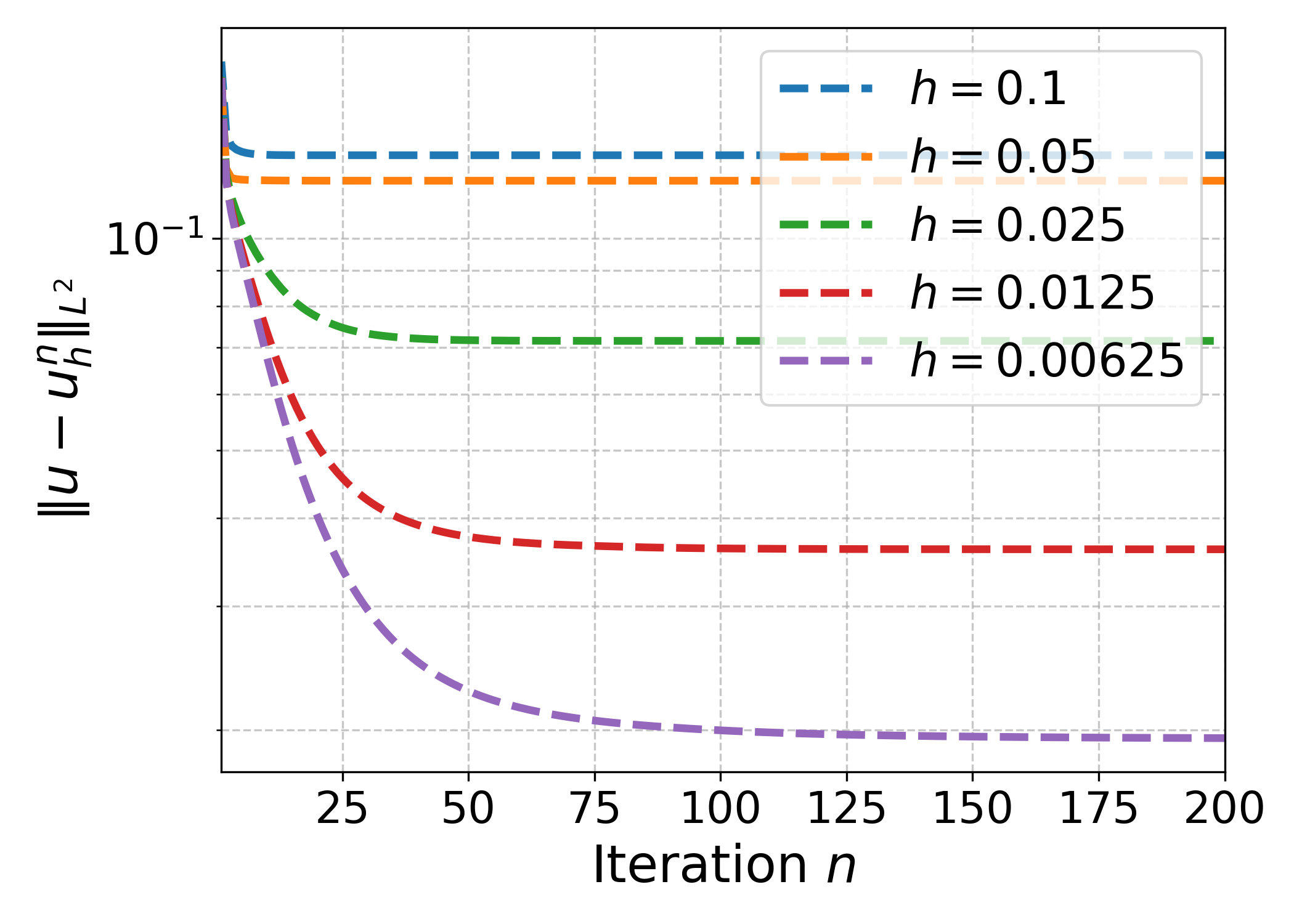}
\end{subfigure}
\caption{Fourth test problem. Left: errors vs. $h$. Right: $\|u-u_h^n \|_{L^2(\Omega)}$ vs. splitting iterations for different values of $h$.}
\label{fig:root_central_iter}
\end{figure}

\subsection{Numerical results on adapted meshes}\label{ssec:adapt}

We now revisit the numerical experiments presented in \Cref{ssec:nonadapt} to evaluate the performance of the \( H^1 \)-error indicator \( \hat{\eta} \), defined as  
\[
\hat{\eta} := \left( \sum_{K \in \mathcal{T}_h} \hat{\eta}_K^2 \right)^{1/2},
\]  
where each local indicator \( \hat{\eta}_K \) is given by \eqref{eq:etahatK}. According to the numerical results in \Cref{tab:errormin}, the error \( \|\omega - \omega_h\|_{H^{-1}(\Omega)} \) exhibits a convergence rate faster than \( \mathcal{O}(h) \). This suggests that \( \hat{\eta} \) is an appropriate indicator for the \( H^1 \) error, as the contributions from other terms in \eqref{eq:h1erru} are comparatively negligible. The goal of the adaptive algorithm is to generate a sequence of meshes such that the relative estimated error remains close to a prescribed tolerance \(\mathrm{TOL} \), i.e.,  
\begin{equation*}\label{eq:tol}
0.75\, \mathrm{TOL} \leq  \dfrac{\hat{\eta}}{\|\nabla u_h\|_{L^2(\Omega)}} \leq 1.25\, \mathrm{TOL}.
\end{equation*}
To satisfy the condition \eqref{eq:tol}, it is sufficient to ensure that, for all \( K \in \mathcal{T}_h \),
\[
\frac{0.75^2\, \mathrm{TOL}^2\, \|\nabla u_h\|^2_{L^2(\Omega)}}{N_K} \leq \hat{\eta}_K^2 \leq \frac{1.25^2\, \mathrm{TOL}^2\, \|\nabla u_h\|^2_{L^2(\Omega)}}{N_K},
\]
where \( N_K \) denotes the number of elements in the mesh. Starting from a coarse mesh (\( h = 0.1 \)), the cell \( K \) is refined if \( \hat{\eta}_K^2 \) exceeds the upper bound, and coarsened if it falls below the lower bound; otherwise, the mesh remains unchanged. In practice, each mesh refinement step is performed when the condition \(\| u_h^{n+1} - u_h^n \|_{L^2(\Omega)} \leq 10^{-8}\) is met, which typically occurs within fewer than $50$ splitting iterations. To avoid infinite mesh refinement, we also impose the constraint \(\frac{h_{\max}}{h_{\min}} \leq 40\).

\begin{table}[tbp]
\caption{First test problem with adaptation. Error estimators on non-adapted mesh.}
\label{tab:tabnoadapt_exp}
\centering
\begin{tabular}{c c c c c c } 
 $h$ & $N_v$ & $\|u-u_h\|_{L^2(\Omega)}$ &$\|\nabla(u-u_h)\|_{L^2(\Omega)}$ & $\hat{\eta}$ &$\frac{\hat{\eta}}{\|\nabla(u-u_h)\|_{L^2(\Omega)}}$ \\ 
 \hline
 \hline
 $0.05$ & $491$ & $0.2531$ & $3.3705$ & $15.8763$ &$4.7104$  \\ 
 \hline
 $0.025$ & $1904$ & $0.0632$ & $1.4837$ & $7.8397$ &$5.2840$ \\ 
 \hline
 $0.0125$ & $7498$ & $0.0162$ & $0.6920$ & $3.9149$ &$5.6578$\\ 
\end{tabular}
\end{table}
\begin{table}[tbp]
\caption{First test problem with adaptation. Error estimators on adapted mesh.}
\label{tab:tabadapt_exp}
\centering
\begin{tabular}{c c c c c c } 
 $\mathrm{TOL}$ & $N_v$ & $\|u-u_h\|_{L^2(\Omega)}$ &$\|\nabla(u-u_h)\|_{L^2(\Omega)}$ & $\hat{\eta}$ &$\frac{\hat{\eta}}{\|\nabla(u-u_h)\|_{L^2(\Omega)}}$ \\ 
 \hline
 \hline
 $1$ & $51$ & $0.5433$ & $7.2043$ & $36.5689$ &$5.0760$  \\ 
 \hline
 $0.5$ & $157$ & $0.2810$ & $3.4795$ & $16.8033$ &$4.8293$  \\ 
 \hline
 $0.25$ & $748$ & $0.1277$ & $1.5232$ & $7.8739$ &$5.1692$ \\ 
\end{tabular}
\end{table}

\subsubsection{First test case with adaptation}
As a first example, we consider a variation of the example in \Cref{sssec:example1}. Specifically, we take $\Omega:=[0,1]^2$ and \( u(x_1, x_2) = e^{2(x_1^2 + x_2^2)} \), which exhibits a steep gradient near the corner \( (1,1) \). \Cref{tab:tabnoadapt_exp} shows the \( L^2 \) and \( H^1 \) error norms, along with the error indicator for the \( H^1 \) norm on a non-adapted mesh. The effectivity index, defined as \(e_i := \hat{\eta}/\|\nabla(u - u_h)\|_{L^2(\Omega)}\), stabilizes around a value of 5. \Cref{tab:tabadapt_exp} reports the results of the mesh adaptivity algorithm for different values of $\mathrm{TOL}$. In this case as well, the effectivity index \( e_i \) stabilizes around 5 and the error halves when the tolerance $\mathrm{TOL}$ is halved. Moreover, we observe that the mesh is appropriately refined near the corner \( (1,1) \) (see \Cref{fig:exp_adapt}), and that the adapted mesh achieves a smaller error in the $H^1$ norm with a lower number of vertices (see \Cref{tab:tabnoadapt_exp,tab:tabadapt_exp}).

\begin{figure}[tbp]
\centering
\begin{subfigure}[t]{0.32\linewidth}
\centering
\includegraphics[width=0.9\linewidth]{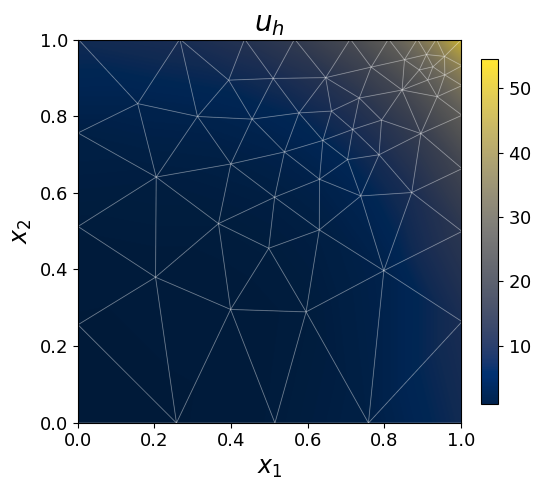}
\caption{$\mathrm{TOL}=1$.}
\end{subfigure}
\begin{subfigure}[t]{0.32\linewidth}
\centering
\includegraphics[width=0.9\linewidth]{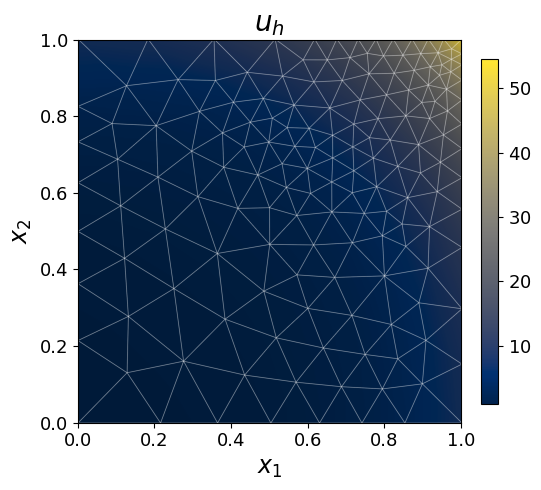}
\caption{$\mathrm{TOL}=0.5$.}
\end{subfigure}
\begin{subfigure}[t]{0.32\linewidth}
\centering
\includegraphics[width=0.9\linewidth]{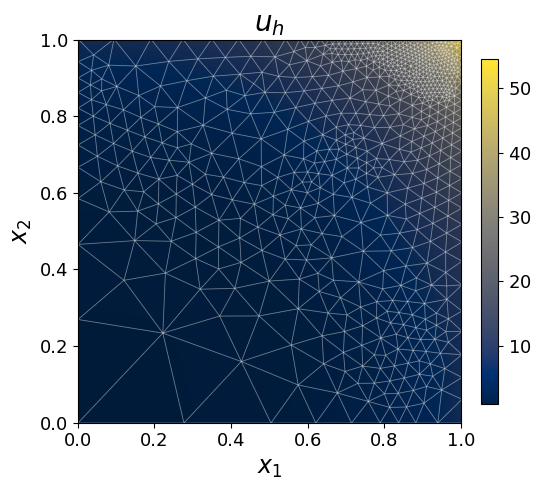}
\caption{$\mathrm{TOL}=0.25$.}
\end{subfigure}
\caption{First test problem with adaptation. Plots of the adapted mesh for $\mathrm{TOL}\in\{1,0.5,0.25\}$.}
\label{fig:exp_adapt}
\end{figure}

\subsubsection{Second test case with adaptation}
We analyze the example presented in the second test case of the previous section (\Cref{sssec:example3}), where $\Omega:=[0,1]^2$ and \( u(x_1, x_2) = -\sqrt{R^2 - (x_1^2 + x_2^2)} \). We begin by considering the case \( R = 2 \). \Cref{tab:tabnoadapt_root4} shows the \( L^2 \) and \( H^1 \) error norms, along with the error indicator for the \( H^1 \) norm on a non-adapted mesh. The effectivity index stabilizes around 7. \Cref{tab:tabadapt_root4} reports the results of the mesh adaptivity algorithm for different values of $\mathrm{TOL}$. The effectivity index remains close to 7, and the \( H^1 \) error is halved when the tolerance is halved. For this value of \( R \), we cannot conclude whether the adapted mesh yields a smaller error. This is likely due to the fact that the solution does not exhibit steep gradients, as in the previous example, and thus a uniform mesh is as appropriate as an adapted one. Next, we consider the limiting case \( R = \sqrt{2} \), to investigate whether the error estimator remains effective when the solution does not belong to \( H^2(\Omega) \). \Cref{tab:tabadapt_rootsqrt2} reports the errors and the value of \( \hat{\eta} \) for the adapted mesh. The effectivity index is approximately 4, and once again, the error is halved when the tolerance is halved. Moreover, compared to the results shown in \Cref{fig:root_sqr2}, where an error in the $H^1$ norm of 10\% could only be achieved with very fine uniform meshes, we now obtain a smaller error using significantly fewer vertices. \Cref{fig:rootsqr2_adapt} shows how the mesh adapts for different values of $\mathrm{TOL}$, with refinement concentrated near the singularity at \( (1,1) \).

\begin{table}[tbp]
\caption{Second test problem with adaptation, $R=2$. Error estimators on non-adapted mesh.}
\label{tab:tabnoadapt_root4}
\centering
\begin{tabular}{c c c c c c } 
 $h$ & $N_v$ & $\|u-u_h\|_{L^2(\Omega)}$ &$\|\nabla(u-u_h)\|_{L^2(\Omega)}$ & $\hat{\eta}$ &$\frac{\hat{\eta}}{\|\nabla(u-u_h)\|_{L^2(\Omega)}}$ \\ 
 \hline
 \hline
 $0.05$ & $491$ & $1.69\cdot 10^{-4}$ & $0.0094$ & $0.0678$ &$7.2333$  \\ 
 \hline
 $0.025$ & $1904$ & $4.15\cdot 10^{-5}$ & $0.0046$ & $0.0338$ &$7.2984$ \\ 
 \hline
 $0.0125$ & $7498$ & $1.14\cdot 10^{-5}$ & $0.0023$ & $0.0168$ &$7.3267$\\ 
\end{tabular}
\end{table}
\begin{table}[tbp]
\caption{Second test problem with adaptation, $R=2$. Error estimators on adapted mesh.}
\label{tab:tabadapt_root4}
\centering
\begin{tabular}{c c c c c c } 
 $\mathrm{TOL}$ & $N_v$ & $\|u-u_h\|_{L^2(\Omega)}$ &$\|\nabla(u-u_h)\|_{L^2(\Omega)}$ & $\hat{\eta}$ &$\frac{\hat{\eta}}{\|\nabla(u-u_h)\|_{L^2(\Omega)}}$ \\ 
 \hline
 \hline
 $0.5$ & $63$ & $0.0065$ & $0.0386$ & $0.2151$ &$5.5668$  \\ 
 \hline
 $0.25$ & $327$ & $0.0017$ & $0.0150$ & $0.0937$ &$6.2371$ \\ 
 \hline
 $0.125$ & $1291$ & $3.33\cdot 10^{-4}$ & $0.0069$ & $0.0479$ &$6.9131$  \\ 
\end{tabular}
\end{table}
\begin{table}[tbp]
\caption{Second test problem with adaptation, $R=\sqrt{2}$. Error estimators on adapted mesh.}
\label{tab:tabadapt_rootsqrt2}
\centering
\begin{tabular}{c c c c c c } 
 $\mathrm{TOL}$ & $N_v$ & $\|u-u_h\|_{L^2(\Omega)}$ &$\|\nabla(u-u_h)\|_{L^2(\Omega)}$ & $\hat{\eta}$ &$\frac{\hat{\eta}}{\|\nabla(u-u_h)\|_{L^2(\Omega)}}$ \\ 
 \hline
 \hline
  $1$ & $40$ & $0.0201$ & $0.1541$ & $0.8313$ &$5.3929$  \\ 
\hline
 $0.5$ & $153$ & $0.0076$ & $0.0797$ & $0.3700$ &$4.6409$  \\ 
 \hline
 $0.25$ & $601$ & $0.0028$ & $0.0445$ & $0.1857$ &$4.1762$ \\ 
\end{tabular}
\end{table}

\begin{figure}[tbp]
\centering
\begin{subfigure}[t]{0.32\linewidth}
\centering
\includegraphics[width=0.9\linewidth]{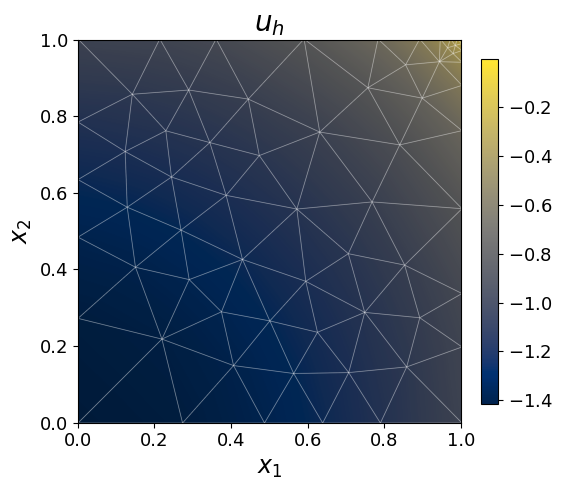}
\caption{$\mathrm{TOL}=1$.}
\end{subfigure}
\begin{subfigure}[t]{0.32\linewidth}
\centering
\includegraphics[width=0.9\linewidth]{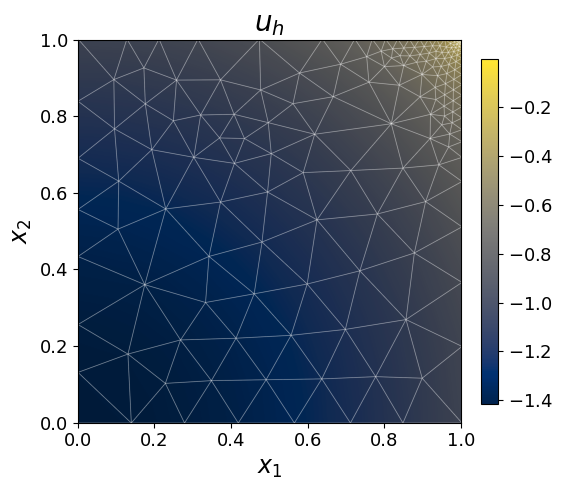}
\caption{$\mathrm{TOL}=0.5$.}
\end{subfigure}
\begin{subfigure}[t]{0.32\linewidth}
\centering
\includegraphics[width=0.9\linewidth]{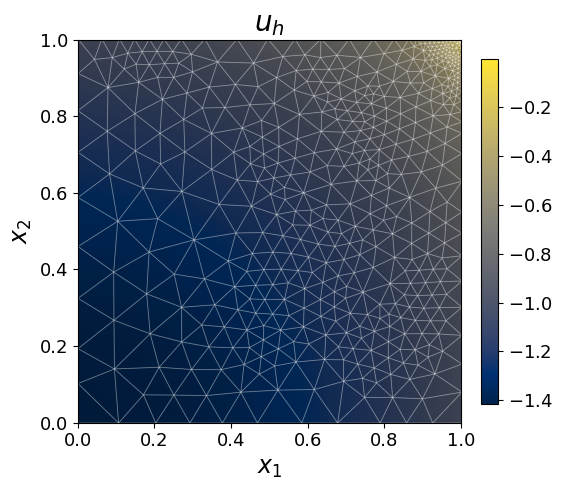}
\caption{$\mathrm{TOL}=0.25$.}
\end{subfigure}
\caption{Second test problem with adaptation, $R=\sqrt{2}$. Plots of the adapted mesh for $\mathrm{TOL}\in\{1,0.5,0.25\}$.}
\label{fig:rootsqr2_adapt}
\end{figure}

\begin{figure}[tbp]
\centering
\begin{subfigure}[t]{0.32\linewidth}
\centering
\includegraphics[width=0.95\linewidth]{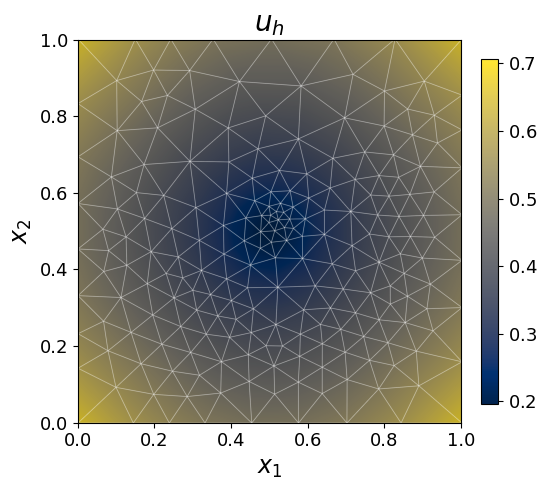}
\caption{$\mathrm{TOL}=0.5$.}
\end{subfigure}
\begin{subfigure}[t]{0.32\linewidth}
\centering
\includegraphics[width=0.95\linewidth]{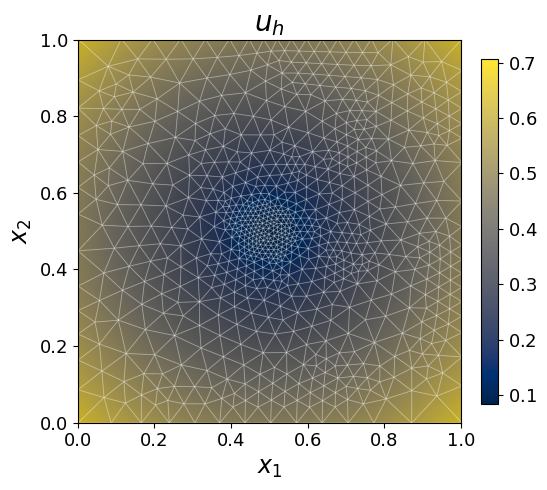}
\caption{$\mathrm{TOL}=0.25$.}
\end{subfigure}
\begin{subfigure}[t]{0.32\linewidth}
\centering
\includegraphics[width=0.95\linewidth]{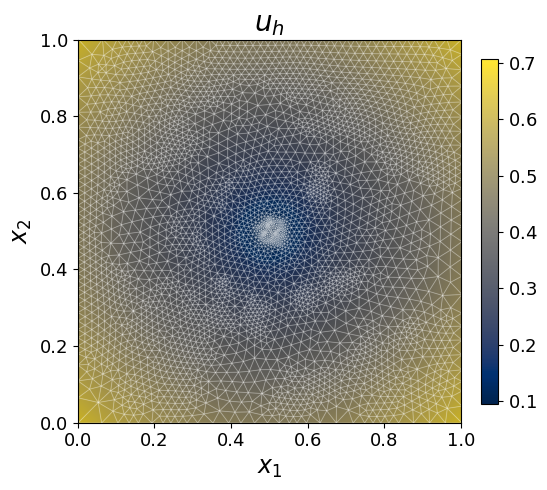}
\caption{$\mathrm{TOL}=0.125$.}
\end{subfigure}
\caption{Third test problem with adaptation. Plots of the adapted mesh for $\mathrm{TOL}\in\{0.5,0.25, 0.125\}$.}
\label{fig:root_central_adapt}
\end{figure}

\begin{table}[tbp]
\caption{Third test problem with adaptation. Error estimators on adapted mesh.}
\label{tab:tabadapt_root_central}
\centering
\begin{tabular}{c c c c c c } 
 $\mathrm{TOL}$ & $N_v$ & $\|u-u_h\|_{L^2(\Omega)}$ &$\|\nabla(u-u_h)\|_{L^2(\Omega)}$ & $\hat{\eta}$ &$\frac{\hat{\eta}}{\|\nabla(u-u_h)\|_{L^2(\Omega)}}$ \\ 
 \hline
 \hline
 $0.5$ & $252$ & $0.1224$ & $0.5540$ & $0.2063$ &$0.3724$  \\ 
 \hline
 $0.25$ & $879$ & $0.0842$ & $0.3809$ &$0.1295$  & $0.3398$\\ 
 \hline
 $0.125$ & $3300$ & $0.0454$ & $0.2068$ & $0.0836$ &$0.4043$ \\
\end{tabular}
\end{table}

\subsubsection{Third test case with adaptation}
The next test problem is the one considered in the last example of the previous section (\Cref{sssec:example5}), featuring a singularity located at the center of the domain. As a result, we expect the adaptive mesh refinement algorithm to concentrate elements around the point $(0.5, 0.5)$. \Cref{fig:root_central_adapt} illustrates the final refined meshes for various tolerance values $\mathrm{TOL}$, and the results confirm this expected behavior. \Cref{tab:tabadapt_root_central} reports the corresponding numerical results. As the tolerance decreases, both the errors in $H^1$  and $L^2$ norms decrease, indicating effective refinement. However, unlike the previous test cases, we observe that halving the tolerance does not necessarily halve the error. This slower convergence rate may be attributed to the fact that the exact solution $u$ does not belong to $H^2(\Omega)$. The effectivity index stabilizes around $0.4$. The fact that it is smaller than $1$ it is not surprising. Indeed, due to the solution’s reduced regularity, other terms in \Cref{lemma:aposteriori} scale like \(\mathcal{O}(h)\), and \(\hat\eta\) captures only a portion of these.

\section{Conclusions}
We have proposed and analyzed an efficient $\mathbb{P}_1$ finite element method for solving a fully nonlinear elliptic problem, building on the nonlinear least-squares splitting algorithm introduced in \cite{caboussat}. By introducing a direct solver for the fourth-order subproblem \eqref{eq:biharmonic}, we achieve a significant reduction in computational cost by approximately an order of magnitude compared to earlier methods. Our approach is supported by both \textit{a priori} and \textit{a posteriori} error estimates, and enhanced by gradient recovery techniques for improved Hessian approximation. Numerical experiments on the unit square validate the theoretical predictions, demonstrating optimal $\mathcal{O}(h)$ convergence in the $H^2$ norm for smooth solutions, a notable advancement over previous work. In non-smooth scenarios, the method remains robust, yielding convergence in the $L^2$ and $H^1$ norms even when classical regularity assumptions fail. The residual-based \textit{a posteriori} estimator effectively guides adaptive mesh refinement, leading to lower errors for the same number of degrees of freedom, with an observed effectivity index close to $5$ in smooth cases.

Future directions include extending the proposed finite element framework and associated error estimates to other fully nonlinear elliptic equations, such as e.g. the Pucci equation. It would also be of interest to generalize the method to different boundary conditions, such as those arising in optimal transport problems, and to consider higher-dimensional domains.

\section*{Acknowledgments}
The authors thank Alexei Lozinski (Universit\'e de Franche-Comt\'e) for fruitful discussions.
\section*{Declarations}
The authors declare that they have no conflicts of interest.

%
%


\bibliographystyle{unsrt}
\bibliography{biblio}

@article{glowinski,
    AUTHOR = {Sorensen, Danny C. and Glowinski, Roland},
     TITLE = {A quadratically constrained minimization problem arising from
              {PDE} of {M}onge-{A}mp\`ere type},
   JOURNAL = {Numer. Algorithms},
  FJOURNAL = {Numerical Algorithms},
    VOLUME = {53},
      YEAR = {2010},
    NUMBER = {1},
     PAGES = {53--66},
      ISSN = {1017-1398,1572-9265},
   MRCLASS = {65K10 (65N99)},
  MRNUMBER = {2566127},
MRREVIEWER = {C.\ Ilioi},
       DOI = {10.1007/s11075-009-9300-5},
       URL = {https://doi.org/10.1007/s11075-009-9300-5},
}

@article {dephilippis,
    AUTHOR = {De Philippis, Guido and Figalli, Alessio},
     TITLE = {The {M}onge-{A}mp\`ere equation and its link to optimal
              transportation},
   JOURNAL = {Bull. Amer. Math. Soc. (N.S.)},
  FJOURNAL = {American Mathematical Society. Bulletin. New Series},
    VOLUME = {51},
      YEAR = {2014},
    NUMBER = {4},
     PAGES = {527--580},
      ISSN = {0273-0979,1088-9485},
   MRCLASS = {35J96 (35B65 35J60 35J66)},
  MRNUMBER = {3237759},
MRREVIEWER = {John\ Urbas},
       DOI = {10.1090/S0273-0979-2014-01459-4},
       URL = {https://doi.org/10.1090/S0273-0979-2014-01459-4},
}

@book {villani,
    AUTHOR = {Villani, C\'edric},
     TITLE = {Optimal transport},
    SERIES = {Grundlehren der mathematischen Wissenschaften [Fundamental
              Principles of Mathematical Sciences]},
    VOLUME = {338},
 PUBLISHER = {Springer-Verlag, Berlin},
      YEAR = {2009},
     PAGES = {xxii+973},
      ISBN = {978-3-540-71049-3},
   MRCLASS = {49-02 (28A75 37J50 49Q20 53C23 58E30)},
  MRNUMBER = {2459454},
MRREVIEWER = {Dario\ Cordero-Erausquin},
       DOI = {10.1007/978-3-540-71050-9},
       URL = {https://doi.org/10.1007/978-3-540-71050-9},
}

@article {bohmer,
    AUTHOR = {B\"ohmer, Klaus},
     TITLE = {On finite element methods for fully nonlinear elliptic
              equations of second order},
   JOURNAL = {SIAM J. Numer. Anal.},
  FJOURNAL = {SIAM Journal on Numerical Analysis},
    VOLUME = {46},
      YEAR = {2008},
    NUMBER = {3},
     PAGES = {1212--1249},
      ISSN = {0036-1429,1095-7170},
   MRCLASS = {35J60 (41A15 47N40 65N30)},
  MRNUMBER = {2390991},
       DOI = {10.1137/040621740},
       URL = {https://doi.org/10.1137/040621740},
}

@article {brenner,
    AUTHOR = {Brenner, Susanne C. and Gudi, Thirupathi and Neilan, Michael
              and Sung, Li-yeng},
     TITLE = {{$\mathcal{C}^0$} penalty methods for the fully nonlinear
              {M}onge-{A}mp\`ere equation},
   JOURNAL = {Math. Comp.},
  FJOURNAL = {Mathematics of Computation},
    VOLUME = {80},
      YEAR = {2011},
    NUMBER = {276},
     PAGES = {1979--1995},
      ISSN = {0025-5718,1088-6842},
   MRCLASS = {65N30 (35J96 65N12)},
  MRNUMBER = {2813346},
MRREVIEWER = {Nilima\ Nigam},
       DOI = {10.1090/S0025-5718-2011-02487-7},
       URL = {https://doi.org/10.1090/S0025-5718-2011-02487-7},
}

@article {brenner2,
    AUTHOR = {Brenner, Susanne C. and Sung, Li-yeng and Tan, Zhiyu and
              Zhang, Hongchao},
     TITLE = {A convexity enforcing {$C^0$} interior penalty method for the
              {M}onge-{A}mp\`ere equation on convex polygonal domains},
   JOURNAL = {Numer. Math.},
  FJOURNAL = {Numerische Mathematik},
    VOLUME = {148},
      YEAR = {2021},
    NUMBER = {3},
     PAGES = {497--524},
      ISSN = {0029-599X,0945-3245},
   MRCLASS = {65N30 (35G30 49M25 49M37 65K10)},
  MRNUMBER = {4293961},
       DOI = {10.1007/s00211-021-01210-x},
       URL = {https://doi.org/10.1007/s00211-021-01210-x},
}

@article {lagrangian,
    AUTHOR = {Dean, Edward J. and Glowinski, Roland},
     TITLE = {An augmented {L}agrangian approach to the numerical solution
              of the {D}irichlet problem for the elliptic {M}onge-{A}mp\`ere
              equation in two dimensions},
   JOURNAL = {Electron. Trans. Numer. Anal.},
  FJOURNAL = {Electronic Transactions on Numerical Analysis},
    VOLUME = {22},
      YEAR = {2006},
     PAGES = {71--96},
      ISSN = {1068-9613},
   MRCLASS = {65N22},
  MRNUMBER = {2208483},
}

@article {caboussat,
    AUTHOR = {Caboussat, Alexandre and Glowinski, Roland and Sorensen, Danny
              C.},
     TITLE = {A least-squares method for the numerical solution of the
              {D}irichlet problem for the elliptic {M}onge-{A}mp\`ere
              equation in dimension two},
   JOURNAL = {ESAIM Control Optim. Calc. Var.},
  FJOURNAL = {ESAIM. Control, Optimisation and Calculus of Variations},
    VOLUME = {19},
      YEAR = {2013},
    NUMBER = {3},
     PAGES = {780--810},
      ISSN = {1292-8119,1262-3377},
   MRCLASS = {65N30 (49M15 65F15 65K10)},
  MRNUMBER = {3092362},
MRREVIEWER = {Jan\ Mandel},
       DOI = {10.1051/cocv/2012033},
       URL = {https://doi.org/10.1051/cocv/2012033},
}

@article {prins,
    AUTHOR = {Prins, C. R. and Beltman, R. and ten Thije Boonkkamp, J. H. M.
              and Ijzerman, W. L. and Tukker, T. W.},
     TITLE = {A least-squares method for optimal transport using the
              {M}onge-{A}mp\`ere equation},
   JOURNAL = {SIAM J. Sci. Comput.},
  FJOURNAL = {SIAM Journal on Scientific Computing},
    VOLUME = {37},
      YEAR = {2015},
    NUMBER = {6},
     PAGES = {B937--B961},
      ISSN = {1064-8275,1095-7197},
   MRCLASS = {65K10 (35J20 35J25 35J96 49K20 49Q20 65N99)},
  MRNUMBER = {3427051},
MRREVIEWER = {Gerd\ Wachsmuth},
       DOI = {10.1137/140986414},
       URL = {https://doi.org/10.1137/140986414},
}

@incollection {yadav,
    AUTHOR = {Yadav, Nitin. K. and ten Thije Boonkkamp, Johannes H. M. and Ijzerman,
              Willem L.},
     TITLE = {A least-squares method for a {M}onge-{A}mp\`ere equation with
              non-quadratic cost function applied to optical design},
 BOOKTITLE = {Numerical mathematics and advanced applications---{ENUMATH} 2017},
    SERIES = {Lect. Notes Comput. Sci. Eng.},
    VOLUME = {126},
     PAGES = {301--309},
 PUBLISHER = {Springer, Cham},
      YEAR = {2019},
      ISBN = {978-3-319-96415-7; 978-3-319-96414-0},
   MRCLASS = {65N21 (49N45 65K05)},
  MRNUMBER = {3976987},
       DOI = {10.1007/978-3-319-96415-7\_26},
       URL = {https://doi.org/10.1007/978-3-319-96415-7_26},
}

@article {dimitrios,
    AUTHOR = {Caboussat, Alexandre and Glowinski, Roland and Gourzoulidis,
              Dimitrios},
     TITLE = {A least-squares/relaxation method for the numerical solution
              of the three-dimensional elliptic {M}onge-{A}mp\`ere equation},
   JOURNAL = {J. Sci. Comput.},
  FJOURNAL = {Journal of Scientific Computing},
    VOLUME = {77},
      YEAR = {2018},
    NUMBER = {1},
     PAGES = {53--78},
      ISSN = {0885-7474,1573-7691},
   MRCLASS = {65N30 (35J96 65K10 65N22)},
  MRNUMBER = {3850345},
MRREVIEWER = {Rajen\ Kumar\ Sinha},
       DOI = {10.1007/s10915-018-0698-6},
       URL = {https://doi.org/10.1007/s10915-018-0698-6},
}

@article{feng,
author = {Feng, Xiaobing and Glowinski, Roland and Neilan, Michael},
title = {Recent Developments in Numerical Methods for Fully Nonlinear Second Order Partial Differential Equations},
journal = {SIAM Review},
volume = {55},
number = {2},
pages = {205-267},
year = {2013},
doi = {10.1137/110825960},

URL = { 
    
        https://doi.org/10.1137/110825960
    
    

},
eprint = { 
    
        https://doi.org/10.1137/110825960
    
    

}

}

@article {neilan,
    AUTHOR = {Feng, Xiaobing and Neilan, Michael},
     TITLE = {Vanishing moment method and moment solutions for fully
              nonlinear second order partial differential equations},
   JOURNAL = {J. Sci. Comput.},
  FJOURNAL = {Journal of Scientific Computing},
    VOLUME = {38},
      YEAR = {2009},
    NUMBER = {1},
     PAGES = {74--98},
      ISSN = {0885-7474,1573-7691},
   MRCLASS = {65N30 (35J60)},
  MRNUMBER = {2472219},
MRREVIEWER = {Igor\ Bock},
       DOI = {10.1007/s10915-008-9221-9},
       URL = {https://doi.org/10.1007/s10915-008-9221-9}
}

@article {lsjac,
    AUTHOR = {Caboussat, Alexandre and Glowinski, Roland and Gourzoulidis,
              Dimitrios},
     TITLE = {A least-squares method for the solution of the non-smooth
              prescribed {J}acobian equation},
   JOURNAL = {J. Sci. Comput.},
  FJOURNAL = {Journal of Scientific Computing},
    VOLUME = {93},
      YEAR = {2022},
    NUMBER = {1},
      ISSN = {0885-7474,1573-7691},
   MRCLASS = {65N30 (35F30 49M20 65K10)},
  MRNUMBER = {4475302},
MRREVIEWER = {Hong-Bo\ Guan},
       DOI = {10.1007/s10915-022-01968-8},
       URL = {https://doi.org/10.1007/s10915-022-01968-8}
}

@article {lsorthogonal,
    AUTHOR = {Caboussat, Alexandre and Gourzoulidis, Dimitrios and Picasso,
              Marco},
     TITLE = {An anisotropic adaptive method for the numerical approximation
              of orthogonal maps},
   JOURNAL = {J. Comput. Appl. Math.},
  FJOURNAL = {Journal of Computational and Applied Mathematics},
    VOLUME = {407},
      YEAR = {2022},
      ISSN = {0377-0427,1879-1778},
   MRCLASS = {65N30 (35F30 49M20 49M41 65K10 65N50)},
  MRNUMBER = {4369636},
MRREVIEWER = {Nicolae\ Pop},
       DOI = {10.1016/j.cam.2021.113997},
       URL = {https://doi.org/10.1016/j.cam.2021.113997}
}

@article {hessianrecovery,
    AUTHOR = {Picasso, Marco and Alauzet, Fr\'ed\'eric and Borouchaki,
              Houman and George, Paul-Louis},
     TITLE = {A numerical study of some {H}essian recovery techniques on
              isotropic and anisotropic meshes},
   JOURNAL = {SIAM J. Sci. Comput.},
  FJOURNAL = {SIAM Journal on Scientific Computing},
    VOLUME = {33},
      YEAR = {2011},
    NUMBER = {3},
     PAGES = {1058--1076},
      ISSN = {1064-8275,1095-7197},
   MRCLASS = {65N30 (65N12 65N15)},
  MRNUMBER = {2800564},
MRREVIEWER = {Javier\ de Frutos},
       DOI = {10.1137/100798715},
       URL = {https://doi.org/10.1137/100798715},
}

@article {lakkis,
    AUTHOR = {Lakkis, Omar and Pryer, Tristan},
     TITLE = {A finite element method for nonlinear elliptic problems},
   JOURNAL = {SIAM J. Sci. Comput.},
  FJOURNAL = {SIAM Journal on Scientific Computing},
    VOLUME = {35},
      YEAR = {2013},
    NUMBER = {4},
     PAGES = {A2025--A2045},
      ISSN = {1064-8275,1095-7197},
   MRCLASS = {65N30 (35J60 35J96 65N12)},
  MRNUMBER = {3085125},
MRREVIEWER = {Marius\ Ghergu},
       DOI = {10.1137/120887655},
       URL = {https://doi.org/10.1137/120887655},
}

@article{bartels,
    author = {Bartels, S{\"o}ren and Tscherner, Philipp},
    title = {Necessary and sufficient conditions for avoiding {B}abu\v{s}ka’s paradox on simplicial meshes},
    journal = {IMA Journal of Numerical Analysis},
    year = {2024},
    month = {08},
    issn = {0272-4979},
    doi = {10.1093/imanum/drae050},
    url = {https://doi.org/10.1093/imanum/drae050},
    eprint = {https://academic.oup.com/imajna/advance-article-pdf/doi/10.1093/imanum/drae050/58921729/drae050.pdf},
}

@article {sweers,
    AUTHOR = {Sweers, Guido},
     TITLE = {A survey on boundary conditions for the biharmonic},
   JOURNAL = {Complex Var. Elliptic Equ.},
  FJOURNAL = {Complex Variables and Elliptic Equations. An International
              Journal},
    VOLUME = {54},
      YEAR = {2009},
    NUMBER = {2},
     PAGES = {79--93},
      ISSN = {1747-6933,1747-6941},
   MRCLASS = {35J40 (35J35)},
  MRNUMBER = {2499118},
       DOI = {10.1080/17476930802657640},
       URL = {https://doi.org/10.1080/17476930802657640},
}

@book {ciarlet,
    AUTHOR = {Ciarlet, Philippe G.},
     TITLE = {The finite element method for elliptic problems},
    SERIES = {Classics in Applied Mathematics},
    VOLUME = {40},
 PUBLISHER = {Society for Industrial and Applied Mathematics (SIAM),
              Philadelphia, PA},
      YEAR = {2002},
     PAGES = {xxviii+530},
      ISBN = {0-89871-514-8},
   MRCLASS = {65N30 (65-02 65N15 74G15 74S05)},
  MRNUMBER = {1930132},
       DOI = {10.1137/1.9780898719208},
       URL = {https://doi.org/10.1137/1.9780898719208},
}

@article {nagazhang,
    AUTHOR = {Zhang, Zhimin and Naga, Ahmed},
     TITLE = {A new finite element gradient recovery method:
              superconvergence property},
   JOURNAL = {SIAM J. Sci. Comput.},
  FJOURNAL = {SIAM Journal on Scientific Computing},
    VOLUME = {26},
      YEAR = {2005},
    NUMBER = {4},
     PAGES = {1192--1213},
      ISSN = {1064-8275,1095-7197},
   MRCLASS = {65N30 (65N12 65N15 74S05)},
  MRNUMBER = {2143481},
MRREVIEWER = {Mohammad\ Asadzadeh},
       DOI = {10.1137/S1064827503402837},
       URL = {https://doi.org/10.1137/S1064827503402837},
}

@book {brenner_fem,
    AUTHOR = {Brenner, Susanne C. and Scott, L. Ridgway},
     TITLE = {The mathematical theory of finite element methods},
    SERIES = {Texts in Applied Mathematics},
    VOLUME = {15},
   EDITION = {Third},
 PUBLISHER = {Springer, New York},
      YEAR = {2008},
     PAGES = {xviii+397},
      ISBN = {978-0-387-75933-3},
   MRCLASS = {65-01 (65-02)},
  MRNUMBER = {2373954},
       DOI = {10.1007/978-0-387-75934-0},
       URL = {https://doi.org/10.1007/978-0-387-75934-0},
}

@article {ainsworth,
    AUTHOR = {Ainsworth, Mark and Oden, J. Tinsley},
     TITLE = {A posteriori error estimation in finite element analysis},
   JOURNAL = {Comput. Methods Appl. Mech. Engrg.},
  FJOURNAL = {Computer Methods in Applied Mechanics and Engineering},
    VOLUME = {142},
      YEAR = {1997},
    NUMBER = {1-2},
     PAGES = {1--88},
      ISSN = {0045-7825,1879-2138},
   MRCLASS = {65N30 (73V05 76M10)},
  MRNUMBER = {1442375},
MRREVIEWER = {Rolf\ Rannacher},
       DOI = {10.1016/S0045-7825(96)01107-3},
       URL = {https://doi.org/10.1016/S0045-7825(96)01107-3},
}

@article{dimitrios_adaptive,
     author = {Alexandre Caboussat and Dimitrios Gourzoulidis and Marco Picasso},
     title = {An adaptive least-squares algorithm for the elliptic {M}onge-{A}mp\`ere equation},
     journal = {Comptes Rendus. M\'ecanique},
     pages = {277--292},
     publisher = {Acad\'emie des sciences, Paris},
     volume = {351},
     number = {S1},
     year = {2023},
     doi = {10.5802/crmeca.222},
     language = {en},
}

@techreport{bl2d,
  TITLE = {{BL2D-V2 : mailleur bidimensionnel adaptatif}},
  AUTHOR = {Laug, Patrick and Borouchaki, Houman},
  URL = {https://inria.hal.science/inria-00069902},
  TYPE = {Research Report},
  NUMBER = {RT-0275},
  PAGES = {49},
  INSTITUTION = {{INRIA}},
  YEAR = {2003},
  MONTH = Jan,
}

@book {gutierrez,
    AUTHOR = {Guti\'errez, Cristian E.},
     TITLE = {The {M}onge-{A}mp\`ere equation},
    SERIES = {Progress in Nonlinear Differential Equations and their
              Applications},
    VOLUME = {89},
   EDITION = {Second},
 PUBLISHER = {Birkh\"auser/Springer},
      YEAR = {2016},
     PAGES = {xiv+216}
}

@book {johnlee,
    AUTHOR = {Lee, John M.},
     TITLE = {Introduction to smooth manifolds},
    SERIES = {Graduate Texts in Mathematics},
    VOLUME = {218},
   EDITION = {Second},
 PUBLISHER = {Springer, New York},
      YEAR = {2013},
     PAGES = {xvi+708},
      ISBN = {978-1-4419-9981-8},
   MRCLASS = {58-01 (53-01 57-01)},
  MRNUMBER = {2954043},
}

@article {Leobacher,
    AUTHOR = {Leobacher, Gunther and Steinicke, Alexander},
     TITLE = {Existence, uniqueness and regularity of the projection onto
              differentiable manifolds},
   JOURNAL = {Ann. Global Anal. Geom.},
  FJOURNAL = {Annals of Global Analysis and Geometry},
    VOLUME = {60},
      YEAR = {2021},
    NUMBER = {3},
     PAGES = {559--587},
      ISSN = {0232-704X,1572-9060},
   MRCLASS = {57N40 (53A07)},
  MRNUMBER = {4304862},
MRREVIEWER = {Qianqian\ Xia},
       DOI = {10.1007/s10455-021-09788-z},
       URL = {https://doi.org/10.1007/s10455-021-09788-z},
}

@article {nagazhang2,
    AUTHOR = {Naga, Ahmed and Zhang, Zhimin},
     TITLE = {A posteriori error estimates based on the polynomial
              preserving recovery},
   JOURNAL = {SIAM J. Numer. Anal.},
  FJOURNAL = {SIAM Journal on Numerical Analysis},
    VOLUME = {42},
      YEAR = {2004},
    NUMBER = {4},
     PAGES = {1780--1800},
      ISSN = {0036-1429,1095-7170},
MRREVIEWER = {Wen\ Bin\ Liu},

}

@article {pucci,
    AUTHOR = {Caboussat, Alexandre},
     TITLE = {A least-squares/relaxation method for the numerical solution
              of a 2{D} {P}ucci's equation},
   JOURNAL = {Methods Appl. Anal.},
  FJOURNAL = {Methods and Applications of Analysis},
    VOLUME = {26},
      YEAR = {2019},
    NUMBER = {2},
     PAGES = {113--132},
      ISSN = {1073-2772,1945-0001},
   MRCLASS = {65N22 (35J60 49M15 49M41 65N30)},
}

@article{Morley1968,
 title={The Triangular Equilibrium Element in the Solution of Plate Bending Problems},
  volume={19},
  DOI={10.1017/S0001925900004546}, 
  number={2}, journal={Aeronautical Quarterly},
  author={Morley, L. S. D.}, 
  year={1968}, 
  pages={149–169}
 }

@incollection {ciarletraviart,
    AUTHOR = {Ciarlet, P. G. and Raviart, P.-A.},
     TITLE = {A mixed finite element method for the biharmonic equation},
 BOOKTITLE = {Mathematical aspects of finite elements in partial
              differential equations ({P}roc. {S}ympos., {M}ath. {R}es.
              {C}enter, {U}niv. {W}isconsin, {M}adison, {W}is., 1974)},
     PAGES = {125--145},
 PUBLISHER = {Academic Press, New York-London},
      YEAR = {1974},
   MRCLASS = {65N30},
  MRNUMBER = {657977},
}

@article {monk,
    AUTHOR = {Monk, Peter},
     TITLE = {A mixed finite element method for the biharmonic equation},
   JOURNAL = {SIAM J. Numer. Anal.},
  FJOURNAL = {SIAM Journal on Numerical Analysis},
    VOLUME = {24},
      YEAR = {1987},
    NUMBER = {4},
     PAGES = {737--749},
      ISSN = {0036-1429},
   MRCLASS = {65N30},
  MRNUMBER = {899701},
MRREVIEWER = {William\ W.\ Hager},
       DOI = {10.1137/0724048},
       URL = {https://doi.org/10.1137/0724048},
}

@article {scholz,
    AUTHOR = {Scholz, Reinhard},
     TITLE = {A mixed method for 4th order problems using linear finite
              elements},
   JOURNAL = {RAIRO Anal. Num\'er.},
  FJOURNAL = {RAIRO Analyse Num\'erique},
    VOLUME = {12},
      YEAR = {1978},
    NUMBER = {1},
     PAGES = {85--90, iii},
      ISSN = {0399-0516,0516-2777},
   MRCLASS = {65N30},
  MRNUMBER = {483557},
MRREVIEWER = {P.\ G.\ Ciarlet},
       DOI = {10.1051/m2an/1978120100851},
       URL = {https://doi.org/10.1051/m2an/1978120100851},
}

@article {brennerbih,
    AUTHOR = {Brenner, Susanne C. and Sung, Li-Yeng},
     TITLE = {{$C^0$} interior penalty methods for fourth order elliptic
              boundary value problems on polygonal domains},
   JOURNAL = {J. Sci. Comput.},
  FJOURNAL = {Journal of Scientific Computing},
    VOLUME = {22/23},
      YEAR = {2005},
     PAGES = {83--118},
      ISSN = {0885-7474,1573-7691},
   MRCLASS = {65N30},
  MRNUMBER = {2142191},
       DOI = {10.1007/s10915-004-4135-7},
}

@article {invalidity,
    AUTHOR = {Zhang, Sheng and Zhang, Zhimin},
     TITLE = {Invalidity of decoupling a biharmonic equation to two
              {P}oisson equations on non-convex polygons},
   JOURNAL = {Int. J. Numer. Anal. Model.},
  FJOURNAL = {International Journal of Numerical Analysis and Modeling},
    VOLUME = {5},
      YEAR = {2008},
    NUMBER = {1},
     PAGES = {73--76},
      ISSN = {1705-5105,2617-8710},
   MRCLASS = {35J40},
  MRNUMBER = {2479119},
MRREVIEWER = {R.\ Kodn\'ar},
}

@article {sweers_corner,
    AUTHOR = {Gerasimov, Tymofiy and Stylianou, Athanasios and Sweers,
              Guido},
     TITLE = {Corners give problems when decoupling fourth order equations
              into second order systems},
   JOURNAL = {SIAM J. Numer. Anal.},
  FJOURNAL = {SIAM Journal on Numerical Analysis},
    VOLUME = {50},
      YEAR = {2012},
    NUMBER = {3},
     PAGES = {1604--1623},
      ISSN = {0036-1429,1095-7170},
   MRCLASS = {35J40 (74G15 74K20)},
  MRNUMBER = {2970757},
MRREVIEWER = {Joel\ Kilty},
       DOI = {10.1137/100806151},
       URL = {https://doi.org/10.1137/100806151},
}

@article {guo2018,
    AUTHOR = {Guo, Hailong and Zhang, Zhimin and Zou, Qingsong},
     TITLE = {A {$C^0$} linear finite element method for biharmonic
              problems},
   JOURNAL = {J. Sci. Comput.},
  FJOURNAL = {Journal of Scientific Computing},
    VOLUME = {74},
      YEAR = {2018},
    NUMBER = {3},
     PAGES = {1397--1422},
      ISSN = {0885-7474,1573-7691},
   MRCLASS = {65N30 (65K15 65N12)},
  MRNUMBER = {3767813},
MRREVIEWER = {Pierre\ Gosselet},
       DOI = {10.1007/s10915-017-0501-0},
       URL = {https://doi.org/10.1007/s10915-017-0501-0},
}

@article {zhang_rec,
    AUTHOR = {Guo, Hailong and Zhang, Zhimin and Zhao, Ren},
     TITLE = {Hessian recovery for finite element methods},
   JOURNAL = {Math. Comp.},
  FJOURNAL = {Mathematics of Computation},
    VOLUME = {86},
      YEAR = {2017},
    NUMBER = {306},
     PAGES = {1671--1692},
      ISSN = {0025-5718,1088-6842},
   MRCLASS = {65N30 (65N15 65N50)},
  MRNUMBER = {3626532},
MRREVIEWER = {JiChun\ Li},
       DOI = {10.1090/mcom/3186},
       URL = {https://doi.org/10.1090/mcom/3186},
}

\end{document}